\documentclass[a4paper, 11pt]{article}

\usepackage{german, ngerman}
\usepackage[german]{babel}
\usepackage[latin1]{inputenc}

\usepackage{dsfont}
\usepackage{amssymb}
\usepackage{amsfonts}
\usepackage{amsthm}
\usepackage{amsmath}
\usepackage{verbatim}
\usepackage{a4wide}
\usepackage{bm}
\usepackage{color}
\usepackage{appendix}
\usepackage{amssymb}
\usepackage{bbm}
\usepackage{bm}
\usepackage{fancybox}
\usepackage{fancyhdr}

\newtheorem{lemma}{Lemma}[section]
\newtheorem{theorem}{Theorem}[section]

\theoremstyle{definition}
\newtheorem{definition}[lemma]{Definition}
\theoremstyle{definition}
\newtheorem{remark}[lemma]{Remark}
\theoremstyle{definition}

{\catcode `\@=11 \global\let\AddToReset=\@addtoreset}
\AddToReset{equation}{section}

\newcommand{\veps}{{\varepsilon }}
\newcommand{\bx}{\bm x}
\newcommand{\bv}{\bm v}

\newcommand{\R}{\mathbb{R}}

\newcommand{\diver}{{\rm div}}

\newcommand{\bolds}{\bm s}
\newcommand{\bnu}{\bm \nu}
\newcommand{\Chi}{ X}

\newcommand{\Mu}{\mathcal M}
\newcommand{\del}{\partial}

\newcommand{\bV}{\bm V}
\newcommand{\wnu}{w_{\bnu}}
\newcommand{\sns}{{\bs \sigma}^\textrm{NS}}

\newcommand{\alphasum}{\sum_{\alpha=1}^{N}}
\newcommand{\betasum}{\sum_{\beta=1}^{N-1}}

\newcommand\jump[1]{{\ensuremath{[\![ #1 ]\!] }}}

\newcommand{\dx}{\operatorname{d}\bx}
\newcommand{\dz}{\operatorname{d}z}

\renewcommand{\div}{\textrm{div}}
\newcommand{\bs}[1]{\boldsymbol{#1}}

\newcommand{\pl}{\partial}
\newcommand{\eps}{\varepsilon}
\newcommand{\pt}{\partial_t}
\newcommand{\unity}{{\bf 1}}

\begin{document}
\selectlanguage{english}
\begin{center}
\LARGE
Modeling of compressible electrolytes with phase transition\\[4mm]
\end{center}

\begin{center}
Wolfgang Dreyer\footnote{Weierstrass Institute,  
Mohrenstr. 39, 10117 Berlin, Germany, E-Mail:
 Wolfgang.Dreyer@wias-berlin.de, 
 Christiane.Kraus@wias-berlin.de.}, 
Jan Giesselmann\footnote{University of Stuttgart, Department of
  Mathematics,  Pfaffenwaldring 57, 70569 Stuttgart, Germany, E-Mail:
  giesselmann@ians.uni-stuttgart.de. \\ C.K.~would like to thank the 
DFG Research Center ``Mathematics for Key Technologies'' {\sc Matheon} in Berlin (Germany) for supporting this project.}
, Christiane Kraus$^1$
\end{center}







\begin{abstract}
A  novel thermodynamically consistent diffuse interface model is derived 
for compressible electrolytes with phase transitions. The fluid mixtures 
may consist of $N$ constituents with the phases  liquid and vapor, where 
both phases may coexist. In addition, all constituents may consist of polarizable
and magnetizable matter.  
Our introduced thermodynamically consistent diffuse interface model 
may be regarded as a generalized model 
of Allen--Cahn/Navier--Stokes/Poisson type for multi-component flows with phase 
transitions and electrochemical reactions. 
 For the introduced diffuse interface model, we investigate physically
 admissible sharp interface limits by  matched asymptotic techniques. 
We consider two scaling regimes, i.e.~a non-coupled and a
coupled regime, where the coupling takes place between the smallness parameter 
in the Poisson equation and the width of the interface. 
We recover in the sharp interface limit a generalized 
 Allen-Cahn/Euler/Poisson system for mixtures with electrochemical
 reactions in the bulk phases equipped with admissible interfacial
 conditions. The interfacial conditions  satisfy,  for instance, a
 generalized Gibbs-Thomson law and a dynamic Young--Laplace law.
\end{abstract}
\section{Introduction}
In this study, we propose a model for chemically
reacting viscous fluid mixtures that may develop a transition
between a liquid and a vapor phase. The mixture consists of $N$
constituents  which may consist of polarizable and magnetizable matter. 
The system is described by $N$ partial mass balance
equations, a single equation of balance for the barycentric
momentum and an equation of Poisson type.  To
describe phase transitions, we introduce an artificial phase
field indicating the present phase by assigning the values 1 and
-1 to the liquid and the vapor phase, respectively. Within the
transition layer between two adjacent phases, the phase field
smoothly changes between 1 and -1. However, usually the transition
layers are very thin leading to steep gradients of the phase
field.

This model belongs to the class of diffuse interface models. An
alternative model class, that likewise represents phase
transitions in fluid mixtures, contains sharp interface models.
From the modeling point of view, sharp interface models have a simpler
physical basis than diffuse interface models. For this reason, there arises
always the non-trivial question if 
 the sharp interface limits of a given diffuse model lead to admissible sharp interface models.
 
While diffuse interface models solve partial differential equations in
the transition region, sharp interface models deal with jump conditions
across the interface between the phases. Sometimes the jump
conditions are mixed with geometric partial differential
equations.
\par 
For the isothermal and quasi-static setting of electrodynamics,  
our newly introduced diffuse interface model is 
given by the following system of PDEs for the partial mass densities $\rho_\alpha,$ the barycentric velocity $\bv,$
the phase field parameter $\chi$ and the electrical potential $\varphi,$ where the equation for $\rho_N$ is replaced by the evolution equation
for $\rho=\sum_{\alpha=1}^N \rho_\alpha:$  
\begin{equation*}
\begin{split}
0=&\del_t \rho_\alpha + \diver (\rho_\alpha \bv ) - 
\diver \Bigg( \betasum M_{\alpha \beta} \bigg( \nabla \frac{\mu_\beta -
  \mu_N}{T} 
+ \frac{1}{T}\bigg( \frac{z_\beta e_0}{m_\beta} - \frac{z_N e_0}{m_N}  \bigg)
\nabla \varphi \bigg) \Bigg) \\ 
& - \sum_{i=1}^{N_R} m_\alpha \gamma_\alpha^i M_{\mathrm{r}}^i
\Bigg(1-\exp\bigg( \frac{1}{kT} \sum_{\beta=1}^N m_\beta \gamma_\beta^i
\mu_\beta\bigg)\Bigg) , \qquad \qquad \qquad \alpha=1, ..., N-1,\\
 0=&\del_t \rho + \diver(\rho \bv), \\
0=& \del_t (\rho \bv) + \diver (\rho \bv \otimes \bv) + \nabla \bigg(\alphasum \rho_\alpha \mu_\alpha - \rho f -W - \frac{\gamma}{2} |\nabla \chi|^2\bigg) 
+ \gamma \, \diver(\nabla \chi \otimes \nabla \chi)
 \\&  - \diver (\sns) 
+\veps_0\diver \left((1+ s(\chi)) \left(\frac{1}{2} |\nabla \varphi|^2 \unity - \nabla \varphi \otimes \nabla \varphi \right)\right),\\
0=& \del_t\chi + \bv \cdot \nabla \chi + \frac{\tau}{\rho} \left( W' -\gamma \Delta \chi + \frac{\del (\rho f)}{\del \chi} - \frac{\veps_0}{2} s'(\chi) |\nabla \varphi|^2 \right),\\
0=&\veps_0 \, \diver ((1+s(\chi)) \nabla \varphi) + n^{\textrm{F}},
\end{split}
\end{equation*}
where  $M_{\alpha \beta}$, $M_\textrm{r}^i$  are the mobilities, $\mu_\alpha$ the 
chemical potentials,  $T$ is the temperature, $m_\alpha$ the atomic mass, 
$e_0$ the elementary charge, $\eps_0$ is the vacuum permittivity, the numbers $z_\alpha$ are integers and  
$k$ is the Boltzmann constant. The system is based on the following free 
energy 
\begin{equation*}
 \rho\psi:= W(\chi) + \frac{\gamma}{2} |\nabla \chi|^2 + h(\chi) \rho\psi_{\textrm{L}}(\rho_1,\dots,\rho_N) 
+ (1-h(\chi)) \rho\psi_{\textrm{V}}(\rho_1,\dots,\rho_N) - \frac{\veps_0}{2} s(\chi) |{\bs E}|^2,
\end{equation*}
where $W(\chi):=(\chi-1)^2(\chi+1)^2$, $\rho\psi_{\textrm{L}}, \rho
\psi_{\textrm{V}}$ are the free energy functions of the pure phases, ${\bs E}$
is the electric field and 
$h:\R \rightarrow [0,1]$ is a smooth interpolation function satisfying
\begin{equation*}
 h(z) = \left\{ \begin{array}{cl}
                 1 &\text{for} \quad z \geq 1,\\ 0&\text{for} \quad z \leq -1,
                \end{array}
\right.
\end{equation*}
such that $h'(z) = 0$ for  all $ |z| \geq 1.$ 
Similarly, we assume
\begin{equation*}
s(\chi)= h(\chi) s_{\textrm{L}}
+ (1-h(\chi)) s_{\textrm{V}},
\end{equation*}
where $s_{\textrm{L/V}}$ are the susceptibilities of the pure phases.
For brevity, we define
\begin{equation*}
\begin{split}
 (\rho f)((\rho_\alpha)_\alpha,\chi)&:= (\rho f)(\rho_1,\dots,\rho_N,\chi)\\
&:=  h(\chi) \rho\psi_{\textrm{L}}(\rho_1,\dots,\rho_N) + (1-h(\chi))
\rho\psi_{\textrm{V}}(\rho_1,\dots,\rho_N).
\end{split}
\end{equation*}
By definition, the chemical potentials are given by 
\begin{equation*}
\mu_\alpha := \frac{\del(\rho \psi)}{\del \rho_\alpha} = \frac{\del (\rho f)}{\del \rho_\alpha}.
\end{equation*}
In addition, $n^{\textrm{F}} := e_0\sum_{\alpha=1}^N \frac{z_\alpha}{m_\alpha} \rho_\alpha$ is the free charge density, $\sns$ denotes the Navier-Stokes stress
and $\tau, \gamma$ are (positive) constants.

If electrical effects are neglected, our compressible model for multi-phase flows reduces to an  Allen--Cahn/Navier--Stokes 
type model which has been studied in \cite{DGK14}. Moreover, in case there is only one constituent
(undergoing liquid-vapor phase transitions) this model is quite similar to the model derived by Blesgen
\cite{Ble99}. Blesgen's  model has been investigated analytically in
\cite{Kot12,FPRS10}, where existence of 
 strong local-in-time solutions and weak solutions has been shown. A modified version of Blesgen's model can be found in \cite{Wit11}. 

Related to our work {\it without chemical reactions} are diffuse interface models for incompressible and 
quasi-incompressible fluids.
A diffuse interface model of Navier-Stokes-Cahn-Hilliard type for two
{incompressible}, viscous Newtonian 
fluids, having the same densities, has been introduced by Hohenberg and 
Halperin in \cite{HH}. That model has been modified in several thermodynamically 
consistent ways such that different densities are allowed, see
e.g. \cite{GPV,LoTr, AGG}. For existence results of 
strong local-in-time solutions and 
weak solutions, we refer to \cite{Abels2, Abels3, ADG13}. A diffuse interface model for 
two incompressible constituents which  permits the transfer of mass between 
the phases due to diffusion and phase transitions has been proposed in
\cite{ADGK13,marseille}. The densities of the fluids may be different, which leads to 
quasi--incompressibility of the mixture.

\par The work is organized as follows.  In the upcoming section we derive the thermodynamically consistent model 
for multi-component flows with phase transitions and electrochemical
reactions, which is the main contribution of this work. 
The third section is devoted to the non-dimensiona\-lization, the introduction
of two interesting scaling regimes of the system and the setting of asymptotic analysis.
Finally, in Sections \ref{sec:si} and \ref{sec:nd}, we determine the sharp interface
limits for the two different scaling regimes introduced previously.


\section{The electrolyte model \label{sec:electrolyte_model}}
\textbf{Constituents and phases.} We consider a fluid mixture
consisting of $N$ constituents $A_1,$ $A_2,...,$ $A_N$ indexed by
$\alpha\in\{1,2,...,N\}$. The constituents have (atomic) masses
$(m_\alpha)_{\alpha=1,2,...,N}$ and may be carrier of charges
$(z_\alpha e_0)_{\alpha=1,2,...,N}$. The constant $e_0$  is the
elementary charge and the numbers $z_\alpha$ are positive or
negative integers including the value zero. All constituents may
consist of polarizable and magnetizable matter.

The fluid mixture may exist in the two phases liquid(L) and
vapor(V). The two phases may coexist. In this paper, we describe
the phases in the diffuse interface setting, where the interface
between adjacent liquid and vapor phases is modeled by a thin
layer. Within the layer, certain thermodynamic quantities smoothly
change from values in one phase to different values in the
adjacent phase. However, usually steep gradients occur.

Among the $N$ constituents we have neutral
molecules and positive and negative ions which are the products of
dissociation reactions. There are $N_R$ reactions, indexed by
$i\in\{1,2,...,N_R\}$, of the general type
\begin{equation}\label{C1}
    a_1^i A_1+a_2^i A_2+...+a_N^i A_N\rightleftarrows
    b_1^i A_1+b_2^i A_2+...+b_N^i A_N.
\end{equation}
The constants $(a_\alpha^i)_{\alpha=1,2,...,N}$ and
$(b_\alpha^i)_{\alpha=1,2,...,N}$ are positive integers and
$\gamma_\alpha^i=b_\alpha^i-a_\alpha^i$ denote the stoichiometric
coefficients of the reaction $i$.

\textbf{Basic quantities and basic variables.} Two phase mixtures
can be modeled within three different model classes, denoted by Class I - Class
III. Class I considers as basic variables the number densities
$(n_{\alpha})_{\alpha=1,2,...,N}$ of the constituents, the
barycentric velocity $\bs v$, the temperature $T$ of the mixture, the
electromagnetic field $(\bs E,\bs B)$ and the phase field $\chi$.
The basic variables of Class II are the number densities
$(n_{\alpha})_{\alpha=1,2,...,N}$, the velocities $(\bs
v_{\alpha})_{\alpha=1,2,...,N}$ of the constituents, the
temperature $T$, the electromagnetic field $(\bs E,\bs B)$  and
the phase field $\chi$. Finally, in Class III we have the number
densities $(n_{\alpha})_{\alpha=1,2,...,N}$, the velocities $(\bs
v_{\alpha})_{\alpha=1,2,...,N}$, the temperatures
$(T_{\alpha})_{\alpha=1,2,...,N}$ of the constituents, the
electromagnetic field $(\bs E,\bs B)$ and the phase field $\chi$.
In this study, we choose a description within Class I.

The mixture occupies a region $\Omega\subset \mathbbm{R}^3$.
 At any time $t\geq 0$, the thermodynamic state of the mixture
 is described by $N$ partial mass densities $(\rho_{\alpha})_{\alpha=1,2,...,N}$,
  the barycentric velocity $\bs v$, the temperature $T$ of the
mixture and the electromagnetic field $(\bs E,\bs B)$. These
quantities may be functions of time $t\geq 0$ and space $\bs x =
(x_i)_{i=1,...,3} = (x_1,x_2,x_3)$. However, the magnetic field
and the temperature as variables appear only in the modeling part 
(Section \ref{sec:electrolyte_model}).
Finally, we restrict ourselves to isothermal processes and, moreover,
we ignore magnetic fields so that $\bs B$ is omitted and T appears
only as a constant parameter in the equations. 

In order to
indicate the present phase at $(t,\bs x)$, we introduce the so
called phase field $\chi$ as a further basic variable. The phase
field assumes values in the interval $[-1,1]$ with $\chi = 1$ in
the liquid and $\chi=-1$ in the vapor.

Multiplication of the number densities by $m_\alpha$ and
$e_0z_\alpha$, respectively, gives the partial mass densities and
the partial free charge densities:
\begin{equation}\label{BQ1}
    \rho_\alpha=m_\alpha n_\alpha,\qquad
    n_\alpha^\textrm{F}=e_0 z_\alpha n_\alpha.
\end{equation}
Multiplication of the velocities by $n_\alpha m_\alpha$ and $n_\alpha z_\alpha e_0$,
respectively, gives the mass fluxes and the free currents:
\begin{equation}\label{BQ2}
    \bs {j}_\alpha=\rho_\alpha \bs {v}_\alpha,\qquad
    \bs {j}_\alpha^F=e_0z_\alpha \bs {v}_\alpha n_\alpha.
\end{equation}
The mass density of the mixture and the barycentric velocity are
defined by
\begin{equation}\label{BQ3}
    \rho=\sum\limits_{\alpha=1}^N \rho_\alpha,\qquad
    \bs v=\frac{1}{\rho}\sum\limits_{\alpha=1}^N \rho_\alpha \bs
    v_\alpha.
\end{equation}
Non-convective mass fluxes and non-convective currents are defined
by
\begin{equation}\label{BQ4}
    \bs {J}_\alpha=\rho_\alpha \bs u_\alpha,\qquad
    \bs {J}_\alpha^{\mathrm{F}}=\frac{z_\alpha e_0}{m_\alpha} \bs {J}_\alpha,\quad\textrm{where}\quad
    \bs{u}_\alpha=\bs{v}_\alpha-\bs v
\end{equation}
denotes the diffusion velocity.  The definitions
\eqref{BQ4}$_{1,3}$ imply the identity
\begin{equation}\label{BQ5}
    \sum\limits_{\alpha=1}^N \bs J_\alpha=0~.
\end{equation}
Total free charge density and total free current are calculated by
\begin{equation}\label{BQ5a}
    n^\textrm{F}=\sum\limits_{\alpha=1}^N n^\textrm{F}_\alpha,\qquad
    \bs{j}^\textrm{F}=n^\textrm{F}\bs v+
    \sum\limits_{\alpha=1}^N {\bs J}_\alpha^{\mathrm{F}}
  ~.
\end{equation}

The application of Maxwell's theory to continuous matter shows
that the total electric charge density $n^\textrm{e}$ and the total
electric current $\bs{j}^\textrm{e}$ consist of two additive
contributions. We write
\begin{equation}\label{BQ6}
    n^\textrm{e}=n^\textrm{F}+n^\textrm{P},\qquad
\bs{j}^\textrm{e}=\bs{j}^\textrm{F}+\bs{j}^\textrm{P}.
\end{equation}
Besides free charge densities and free currents there are charge
densities and currents due to polarization and magnetization \cite{Mue85}.
\begin{equation}\label{BQ7}
    n^\textrm{P}=-\textrm{div}(\bs P),\qquad
    \bs{j}^\textrm{P}=\frac{\partial \bs{P}}{\partial
    t}+\textrm{curl}(\bs{P}\times\bs{v}+{\bs M})~,
\end{equation}
where $\bs{P}$ and $\bs{M}$  denote the vectors of polarization
and magnetization, respectively. Polarization embodies phenomena
that are caused by microscopic charges, for example, atomic
dipoles within atoms and molecules. Microscopic currents are
macroscopically represented by the magnetization vector. 

Finally, we introduce the total number density of the mixture and
the atomic fractions of the constituents.
\begin{equation}\label{BQ9}
    n=\sum\limits_{\alpha=1}^N
    n_\alpha,\qquad y_\alpha=\frac{n_\alpha}{n}\quad\textrm{with}
    \quad \sum\limits_{\alpha=1}^N
    y_\alpha=1~.
\end{equation}

\textbf{Equations of balance for matter.} The basic variables are
determined by a coupled system of partial differential equations
relying on the quasi-static Maxwell equations and balance equations
for matter. At first we introduce the balance equations for matter
within the Class I model where we need the partial equations of
balance for the mass of the constituents and the balance equations
for the momentum and energy of the mixture. They read
\begin{eqnarray}
    \pl_t \rho_\alpha+\textrm{div}(\rho_\alpha\bs{v}+\bs J_\alpha)
    &=&r_\alpha, \qquad \alpha=1,2,...,N, \label{EoB1}\\
    \pl_t(\rho\bs v)+\textrm{div}(\rho\bs v\otimes\bs
    v-\bs{\sigma})&=&\rho\bs{b}+\bs k,\label{EoB2}\\
     \pl_t \Big(\rho e+\frac{\rho}{2}|\bs{v}|^2 \Big)+
     \textrm{div} \Big((\rho e+\frac{\rho}{2}|\bs{v}|^2)\bs v+\bs q -\bs
     v\cdot\bs\sigma \Big)&=&\rho\bs b\cdot \bs v+\pi.
     \label{EoB3}
\end{eqnarray}
Moreover, we propose a balance equation for the phase field,
\begin{equation}\label{EoB4}
    \pl_t (\rho\chi)+\textrm{div}(\rho\chi\bs{v}+\bs J_\chi)=\xi_\chi.
\end{equation}
Besides the basic variables and the diffusion fluxes from the last
paragraph there occur new quantities (here): $r_\alpha$ - mass
production of constituent A$_\alpha$, $\bs{\sigma}$ - stress,
$\rho e$ - internal energy density, $\bs{q}$ - heat flux, $\bs
J_\chi$ - non-convective flux of the phase field, $\xi_\chi$ -
phase field production. The force density is decomposed into two
different types: $\rho \bs b$ - force density due to gravitation
and inertia, $\bs k$ - Lorentz force density due to
electromagnetic fields. Likewise the power of force is decomposed
into: $\rho \bs b\cdot\bs v$ - power due to gravitation and
inertia, $\pi$ - power due to Joule heat. In the following, we
neglect the force density $\bs b$ and set $\bs b=0$.

Forward and backward reactions contribute to the mass production
rate of constituent A$_\alpha$. The corresponding reactions rates
$R_\textrm{f}^i$ and $R_\textrm{b}^i$ give the number of forward and backward
reactions per volume and per time. We write
\begin{equation}\label{EoB4c}
    r_\alpha=\sum\limits_{i=1}^{N_\textrm{R}}
    m_\alpha \gamma_\alpha^i(R_\textrm{f}^i-R_\textrm{b}^i).
\end{equation}
The conservation of charge and mass for every single reaction
$i\in\{1,2,...,N_\textrm{R}\}$ reads
\begin{equation}\label{EoB5b}
\sum\limits_{\alpha=1}^{N}
    z_\alpha\gamma_\alpha^i=0 \quad\textrm{and}\quad
    \sum\limits_{\alpha=1}^{N}
    m_\alpha\gamma_\alpha^i=0,\quad\textrm{implying\quad}
\sum\limits_{\alpha=1}^{N}r_\alpha=0.
\end{equation}
The condition \eqref{EoB5b}$_3$ represents the conservation law of
total mass. 

Summing up the partial mass balances \eqref{EoB1}
yields the total mass balance of the mixture, i.e.
\begin{equation}\label{EoB5a}
    \pl_t \rho+\textrm{div}(\rho\bs{v})=0.
\end{equation}
Herein,  the definitions \eqref{BQ3} and the conditions
\eqref{EoB5b}$_3$ and \eqref{BQ5} have been used.

\textbf{A short reminder on Maxwell's equations.} The
determination of the electromagnetic field $(\bs E,\bs B)$ relies
on Maxwell's equations. They can be written as,  \cite{Mue85},
\begin{alignat}{3}
    \pl_t \bs B+\textrm{curl}(\bs E)&=0,\qquad\qquad~ \textrm{div}(\bs
    B)&=&0,\label{A1-1}\\
-\frac{1}{c^2}\pl_t \bs E+\textrm{curl}(\bs B)&=\mu_0 \bs
j^\textrm{e},\qquad\quad \textrm{div}(\bs
    E)&=&\frac{1}{\eps_0}n^\textrm{e}\label{A1-2}.
\end{alignat}
The electric and magnetic constants are related to the speed of
light by $c^2=1/(\eps_0\mu_0)$. The total electric charge density
$n^\textrm{e}$ and the total electric current density $\bs{j}^\textrm{e}$
are given by the representations \eqref{BQ5a}--\eqref{BQ9}.

Suitable multiplications of Maxwell's equations by $\bs E$ and
$\bs B$, respectively, lead to two new equations of balance, viz.
\begin{equation}\label{A1-3}
    \pl_t \bs m^\textrm{e}+\textrm{div}(-\bs\sigma^\textrm{e})=
    -n^\textrm{e}\bs E- \bs j^\textrm{e}\times\bs B,\qquad
    \pl_t e^\textrm{e}+\textrm{div}(\bs q^\textrm{e})=-\bs j^\textrm{e}\cdot\bs E~.
\end{equation}
These equations are interpreted as the equations of balance for
electromagnetic momentum and electromagnetic energy. The
corresponding densities and fluxes have the unique representations
\begin{alignat}{3}
    \bs m^\textrm{e}&=\eps_0\bs E\times\bs B, \qquad\qquad
     &\bs \sigma^\textrm{e}&=\eps_0 \bs E\otimes\bs
     E+\frac{1}{\mu_0}\bs B\otimes\bs B-\frac{1}{2}(\eps_0 |\bs E|^2+\frac{1}{\mu_0}|\bs
     B|^2)\unity,\label{A1-4}\\
     e^\textrm{e}&=\frac{\eps_0}{2}|\bs
    E|^2+\frac{1}{2\mu_0}|\bs B|^2,\quad~ 
   \quad & \bs q^\textrm{e}&=\frac{1}{\mu_0}\bs E\times\bs B.\label{A1-5}
\end{alignat}
The balance equations of the electromagnetic momentum and energy
are now added to the corresponding balance equations of matter. We
obtain the equations of balance for total momentum and total
energy. The postulate that total momentum and total energy both
are conserved quantities implies the identification of the Lorentz
force and its power, viz.
\begin{equation}\label{A1-6}
\bs k=n^\textrm{e}\bs E+ \bs j^\textrm{e}\times\bs B,\qquad
    \pi=\bs j^\textrm{e}\cdot\bs E~.
\end{equation}
\textbf{On the quasi-static setting of electrodynamics.} There is
large confusion in the electrochemical literature about the
quasi-static approximation of Maxwell's equations.  For this
reason, a short discussion of the subject is necessary.

At first we rescale time, space, the magnetic and the electric field and the conductivity $\sigma$ according
to
\begin{equation}\label{A1-7}
t= t_0 \tilde t, \quad \bs x= x_0\tilde{ \bs x}, \quad \bs E= E_0 \tilde{ \bs E}, \quad \bs B = \frac{E_0}{c} \tilde{ \bs B},\quad  \sigma = \sigma_0 \tilde\sigma,
\quad n^{\textrm{F}} = n_0^{\textrm{F}} \tilde{n}^{\textrm{F}}.
\end{equation}
From Ohm's law we know that $\bs j^{\textrm{e}} = \sigma_0 \tilde \sigma E_0
\tilde{\bs E}=: \sigma_0 E_0 \tilde{\bs j}^{\textrm{e}}.$ 
The rescaled magnetic field has the same dimension as the electric
field so that both fields can be compared. Furthermore, we set
$x_0/t_0=v_0$ with $v_0$ as a typical diffusion velocity of matter,
i.e.~we set $n_0^{\textrm{F}}v_0 = \sigma_0 E_0.$
Suppressing the tildes in our notation we obtain
\begin{alignat}{3}
    \frac{v_0}{c}\pl_t \bs B+\textrm{curl}(\bs E)&=0,\qquad\qquad\quad~
   & \textrm{div}(\bs
    B)&=0,\label{A1-8}\\
-\frac{v_0}{c}\pl_t \bs E+\textrm{curl}(\bs B)&=c\mu_0 x_0 \sigma_0\bs
j^\textrm{e},\qquad\quad & \textrm{div}(\bs
    E)&=\frac{\sigma_0 t_0}{\eps_0}n^\textrm{e}\label{A1-9}.
\end{alignat}
The dimensionless quantity $c \mu_0 \sigma_0 x_0$ is of order $1$.
The time derivatives in the two equations (\ref{A1-1})$_1$,
(\ref{A1-2})$_1$ are thus multiplied by the small factor $v_0/c$.
Returning to dimensional quantities, the leading order Maxwell's equations reduce to
\begin{alignat}{3}
    \textrm{curl}(\bs E)&=0,\qquad\qquad\qquad &\textrm{div}(\bs
    B)&=0 ,\label{A1-10}\\
\textrm{curl}(\bs B)&=\mu_0 \bs j^\textrm{e},\qquad\quad~
&\textrm{div}(\bs
    E)&=\frac{1}{\eps_0}n^\textrm{e}\label{A1-11},
\end{alignat}
which we call the quasi-static version of the Maxwell equations.
A similar argument shows that the Lorentz force is given by
\begin{equation}\label{EoB5}
    \bs k=n^\textrm{e} \bs E,
\end{equation}
and its power is as before.
In the quasi-static setting the electric field can be derived from
an electric potential $\varphi$. Thus, we have by \eqref{A1-2}
\begin{equation}\label{EoB6}
    \bs E = -\nabla\varphi,\qquad
    \eps_0\Delta\varphi=-n^\textrm{e},
\end{equation}
and the magnetic field follows from \eqref{A1-11}$_1$.

Note that the modeling part of this paper relies on the full system of
Maxwell's equations, but in the application we will use the
quasi-static setting only.

\textbf{The balance equation for the internal energy.} We form the
scalar product of the momentum balance \eqref{EoB2} with the
velocity $\bs v$ to obtain the balance of the kinetic energy. Then
this balance is subtracted from the energy balance \eqref{EoB3}.
The result is the balance of the internal energy, which can be
written as
\begin{equation}\label{EoB8}
    \pl_t(\rho e)+\textrm{div}(\rho e\bs v+\bs q)=
    \bs\sigma:D(\bs v)+
    \bigg(\sum\limits_{\alpha=1}^N\frac{z_\alpha e_0}{m_\alpha}\bs{J}_\alpha+
    \dot{\bs P}+\bs P \textrm{div}(\bs v)-\bs P\cdot\nabla\bs v+
    \textrm{curl}({\bs M}) \bigg)\cdot  \bm{\mathcal{ E}}.
\end{equation}
Here, $\dot{\bs{P}}=\partial_t\bs{P}+\bs v\cdot\nabla\bs{P}$
indicates the material time derivative of the polarization, and
$\bm{\mathcal{ E}}=\bs E+\bs v\times\bs B$ defines the electromotive
intensity. The right hand side of \eqref{EoB8} represents the
production of internal energy due to mechanical stresses,
diffusion of free charges and polarization and magnetization.

\textbf{Constitutive model, Part 1: General strategy.} We choose
as variables of our model the quantities
$(\rho_{\alpha})_{\alpha=1,2,...,N}$, $\bs v$, $T$, $\chi$, $\bs
E$ and $\bs B$. Their determination relies on (i) the balance
equations for the partial masses \eqref{EoB1}, for the (barycentric)
momentum \eqref{EoB2}, and for the phase field \eqref{EoB4}, (ii)
the internal energy balance \eqref{EoB8}, (iii) Maxwell's
equations \eqref{A1-1} and \eqref{A1-2}.

These equations contain further quantities that are not in the
list of our variables: In the mass balances we have the reaction
rates $R_\textrm{f,b}^i$ and the diffusion fluxes $\bs{J}_\alpha$.
The constitutive quantities of the momentum balance are the stress
$\bs\sigma$, the charge density $n^\textrm{e}$ and the electric
current $\bs{j}^\textrm{e}$. The latter quantities also occur in
Maxwell's equations. The phase field balance contains the phase
field flux $\bs{J}_\chi$ and the production rate $\xi_\chi$.
Finally, the constitutive quantities of the internal energy balance
are the internal energy $\rho e$, the heat flux $\bs{q}$, the
stress $\bs\sigma$, the magnetization $\bs M$  and the polarization $\bs{P}$. The constitutive
quantities must be related to the variables in material dependent
manner, i.e. they must be given by constitutive equations.
Thermodynamically consistent constitutive equations have to
satisfy (i) the \textit{principle of material frame indifference}
and (ii) the \textit{entropy principle}.

\noindent {\bf The principle of material frame indifference} makes
a statement on constitutive functions of objective tensors, viz.
constitutive functions of objective tensors must remain invariant
with respect to {\it Euclidean transformations}.

To introduce these concepts we first consider Euclidean
transformations, which are the most general transformation between
two Cartesian coordinate systems with coordinates written as $(t,
x_1,\dots,x_3)=(t, x_i)_{i=1,\dots,3}$ and $(t^\ast,
x_1^\ast,\dots,x_3^\ast)=(t^\ast, x_i^\ast)_{i=1,\dots,3}$,
respectively:
\begin{equation}
t^\ast = t+a, \quad x_i^\ast = O_{ij}(t) x_j + b_i (t), \quad
\mathbf{O}(t) \mathbf{O}(t)^{\sf T} =\unity.
\end{equation}
Next we define the notion of {\it objective scalars, vectors and
tensors (of rank two)} if their components transform according to
\begin{align}
s^\ast & =\det({\bf O})^p s & & \mbox{ for scalars},\\
v_i^\ast & =\det({\bf O})^p O_{ij} v_j & & \mbox{ for vectors},\\
T_{ij}^\ast & = \det({\bf O})^p O_{ik} O_{jl} T_{kl} & &  \mbox{
for rank two tensors}.
\end{align}
Scalars and vectors are also called tensor of rank 0 and 1,
respectively. For $p=0$ the objective tensor is called absolute
objective tensor and for $p=1$ we have an axial objective tensor.

As an example, we consider an objective absolute tensor $\bs T$,
i.e. we have $T_{ij}^\ast = O_{ik} O_{jl} T_{kl}$. Let us further
assume that ${\bf T}$ is a function of $\nabla {\bf v}$, so that
in general we have
\begin{equation}
T_{ij}=f_{ij} \Big(\frac{\partial v_k}{\partial x_l}\Big), \quad \mbox{
respectively } T_{ij}^\ast=f_{ij}^\ast \Big(\frac{\partial
v_k^\ast}{\partial x_l^\ast}\Big).
\end{equation}
Then, with $v_i^\ast =O_{ij} (v_j +  \dot{O}_{jl} O_{kl}(x_k^\ast
-b_k)-\dot{b}_j)$, objectivity amounts to
\begin{equation}
\label{f-trafo} f^\ast (\mathbf{O}(t) \nabla \mathbf{v}\,
\mathbf{O}(t)^{\sf T} + \dot{\mathbf{O}}(t) \mathbf{O}(t)^{\sf T})
= \mathbf{O}(t) f(\nabla \mathbf{v}) \mathbf{O}(t)^{\sf T}.
\end{equation}
In this case, the principle of material frame indifference states
that $f^\ast=f$. In other words it implies that $f$ is an {\it
isotropic function}. Moreover, $f$ can only depend on the
symmetric part ${\bf D}$ of $\nabla {\bf v}$ which follows from
\eqref{f-trafo} by choosing ${\bf O}=\unity$ and
$\dot{\mathbf{O}}=-{\bf R}$, where ${\bf R}$ is the anti-symmetric
part of $\nabla {\bf v}$.

\noindent {\bf Classification 1: Transformation properties of
important quantities.} In this paragraph, we indicate the
transformation properties of some important quantities. There are
kinematic and non-kinematic quantities. While the transformation
properties of kinematic quantities can be derived, the
transformation properties of non-kinematic quantities must be
postulated. More details and motivations can be found in \cite{Mue85}
and \cite{TT60}. Kinematic quantities are for example the
barycentric velocity and the diffusion velocities. The diffusion
velocities and the symmetric part of the velocity gradient are
absolute objective tensors. The barycentric velocity and the
antisymmetric velocity gradients are non-objective quantities.
Here is a list with properties of important quantities:

Absolute objective scalars: mass densities, internal energy
density, phase field, reaction rates, phase field production.

Absolute objective vectors: diffusion fluxes, phase field flux,
charge potential, Lorentz force, polarization and electromotive
force.

Axial objective vectors: magnetization, magnetic flux density.

Absolute objective tensor: (Cauchy) stress tensor.

\noindent Note that the internal energy density is only an
objective scalar if the stress is symmetric, because then the
antisymmetric part of the velocity gradient in \eqref{EoB8} drops
out and the mechanical power ${\bs \sigma}:D({\bs v})$ is formed
with the symmetric part $D({\bs v})$ of the velocity gradient. The
electric field ${\bs E}$ and the magnetic current potential ${\bs
H}$ are not objective quantities. However, the sum of the
electromagnetic terms in \eqref{EoB8} is an objective scalar.

\noindent {\bf Classification 2: Parity of important quantities.}
The formulation of the 2$^\textrm{nd}$ law of thermodynamics needs
a further classification of the involved quantities, which is
related to their physical dimensions. If the units of time and
electric current, i.e. ''second'' and "ampere", respectively, of a
given quantity $q$ appear such that the sum of their powers is uneven, we assign the factor $-1$
according to $\mathcal{P} q=-1$. If the sum of the powers
of "second" and "ampere" is even
we assign $\mathcal{P} q=+1$. Then the quantity $q$ has negative
and positive parity, respectively. It is to be understood that we
choose the units of the SI system.

For example, the parity of the density of mass, momentum, internal
energy and magnetic indiction are then given by
\begin{equation}
\label{parity-examples} [\rho]=\frac{\rm kg}{\rm m^3} \to+1,\qquad
[\rho {\bf v}]=\frac{\rm kg}{\rm m^2 s}\to -1,\qquad [\rho
e]=\frac{\rm kg}{\rm m\, s^2}\to+1\qquad [{\bf B}]=\frac{\rm
Vs}{\rm m^2}\to-1.
\end{equation}
Evidently, the time derivative of a quantity has the opposite
parity, while spatial derivatives
keep the parity unchanged.\\[1ex]

\noindent {\bf Formulation of the entropy principle.} Any solution
of the above systems of partial differential equations, composed
of \eqref{EoB1}, \eqref{EoB2}, \eqref{EoB4}, \eqref{EoB8},
\eqref{A1-1} and \eqref{A1-2}, is called a {\it thermodynamic
process}. Here, by solutions we just mean functions which satisfy
the balance equations in a local sense. In particular, the value
of a quantity and of its spatial derivatives can be chosen independently.
With this concept, the 2$^{\rm nd}$ law of thermodynamics consists
of four universal and two material dependent axioms. For detailed
motivation and further discussion
 see \cite{BD13}.\\[-2ex]
\begin{enumerate}
\item[(I)] There is an entropy/entropy-flux pair $(\rho s, \bs \Phi)$
as a material dependent quantity, where $\rho s$ is an absolute  objective
scalar and $\bs \Phi$ is an absolute objective vector. The entropy has the
physical dimension ${\rm J}\, {\rm kg}^{-1}\, {\rm K}^{-1}={\rm
m}^2 {\rm s}^{-2} {\rm K}^{-1}$, hence is of positive parity. The
entropy flux and the entropy production $\zeta$ thus have negative
parity.\vspace{0.05in} \item[(II)] The pair $(\rho s,\bs \Phi)$
satisfies the balance equation
\begin{equation}\label{EntB1}
\pt (\rho s) + \div (\rho s {\bf v} +\bs \Phi ) = \zeta.
\end{equation}
\item[(III)]
Any admissible entropy/entropy-flux is such that\\[0.5ex]
    (i) $\zeta$ consists of a sum of binary products according to
    \begin{equation}\label{M15}
    \zeta = \sum_m \mathcal{N}_m \mathcal{P}_m,
    \end{equation}
    where the $\mathcal{N}_m$ denote quantities of negative parity, while $\mathcal{P}_m$
    refers to positive parity.\\[0.25ex]
    (ii) $\mathcal{N}_m \mathcal{P}_m \geq 0$ for all $m$ and for every thermodynamic process.\vspace{0.05in}
\item[(IV)]
     A thermodynamic process where $\zeta =0$ is said to be in thermodynamic equilibrium. This statement is to be understood in a pointwise
    sense; in particular, this must not hold everywhere, i.e.\ thermodynamic equilibrium can be attained locally.\\
    A thermodynamic process is called reversible if $\zeta =0$ everywhere.
    \end{enumerate}
In addition, to these universal axioms, we impose two further ones
which refer to the most general constitutive models we are
interested in. These are:
\begin{enumerate}
 \item[(V)]
There are the following {\it dissipative mechanisms} for fluid
mixtures under consideration: {\it diffusion of mass and charge},
{\it chemical reaction}, {\it viscous flow}, {\it heat
conduction}, {\it phase transition due to diffusion} and {\it
phase transition due to phase production}. Correspondingly, in
equilibrium we have
\begin{equation}\label{entZusatz}
    {\bs u}_\alpha=0,\quad R^i_\textrm{f}=R^i_\textrm{b},\quad
    D(\bs{v})=0,\quad {\bs q}=0,\quad {\bs J}_\chi=0,\quad
    \xi_\chi=0.
\end{equation}
\item[(VI)] For the class of fluid mixtures under consideration,
we restrict the dependence of the entropy according to
\begin{equation}
\label{EntB2} \rho s = \rho \tilde s (\rho e-\bm{\mathcal{\bs E}}\cdot
\bs P,\rho_1, \ldots ,\rho_N, \bm{\mathcal {\bs E}},{\bs B}, \chi,
\nabla \chi),
\end{equation}
where $\rho \tilde s$ is a {\it concave function} which satisfies
the principle of material frame indifference. By means of this
function, we define the {\it (absolute) temperature} $T$, the {\it
chemical potentials} $(\mu_i)_{i=1,2,...,N}$ of the constituents
and the chemical potential of the phases $\mu_\chi$ as
\begin{equation}\label{EntB3}
\frac 1 T := \frac{\partial \rho \tilde s}{\partial (\rho
e-\bm{\mathcal{\bs E}}\cdot \bs P)}, \qquad \frac{\mu_i}{T} := -
\frac{\partial \rho \tilde s}{\partial
\rho_i},\qquad\frac{\mu_\chi}{T} := - \bigg(\frac{\partial \rho \tilde
s}{\partial \chi}-\nabla\cdot \frac{\partial \rho \tilde
s}{\partial \nabla\chi}\bigg) .\vspace{0.1in}
\end{equation}
\end{enumerate}

\noindent {\bf Identification of the entropy production.} To
calculate the entropy production, we introduce the material time
derivative in \eqref{EntB1} and insert the entropy function
\eqref{EntB2}. The intermediate result is
\begin{align}
\zeta=& \frac{1}{T}((\rho e)^.-\dot{\bs P}\cdot \bm{\mathcal {\bs E}} -\bs
P\cdot \bm{\mathcal{\dot{\bs E}}})-
\sum\limits_{\alpha=1}^N\frac{\mu_\alpha}{T}\dot{\rho}_\alpha +
\frac{\partial \rho \tilde s}{\partial  \bm{\mathcal {{\bs
E}}}} \cdot \bm{\mathcal {\dot{\bs E}}} +\frac{\partial \rho \tilde
s}{\partial{\bs B}}\cdot{\dot{\bs B}} + \frac{\partial \rho \tilde
s}{\partial \chi}\dot\chi+\frac{\partial \rho \tilde s}{\partial
\nabla \chi}\cdot({\nabla\chi})^.
\nonumber\\
& + \rho s \, \textrm{div}(\bs v)+\textrm{div}(\bs \Phi)
         ~.\label{UG1}
\end{align}
Next, we eliminate the time derivatives of internal energy, partial
mass densities and phase field by means of the corresponding
balance equations. Moreover, the term $(\nabla\chi)^.$ is
substituted by the identity
\begin{equation}\label{UG2}
    (\nabla\chi)^.=\nabla\dot\chi-\nabla \bs v\cdot\nabla\chi.
\end{equation}
Finally, we use the identity
\begin{equation}\label{UG2a}
    \bm{\mathcal E}\cdot\textrm{curl}({\bs M})=
    -\textrm{div}(\bm{\mathcal E}\times\bs{M})+\bs{M}\cdot\textrm{curl}(\bm{\mathcal  E})
\end{equation}
and substitute $\textrm{curl}(\bm{\mathcal E})$ by a variant
of \eqref{A1-1}$_1$, viz.
\begin{equation}\label{UG2b}
    \textrm{curl}(\bm{\mathcal E})=-\dot{\bs B}-{\bs
    B}\textrm{div}({\bs v})+{\bs B}\cdot\nabla{\bs v}.
\end{equation}
After rearranging terms we obtain
\begin{align}
\zeta= &~ \textrm{div}\left(\bs \Phi-\frac{1}{T}(\bs q+
\bm{\mathcal {\bs
E}}\times{\bs M})+\frac{1}{T}\sum\limits_{\alpha=1}^N\mu_\alpha
\bs{J}_\alpha+ \frac{\partial \rho \tilde s}{\partial
\nabla \chi}\dot\chi-\frac{1}{\rho}\bigg(\frac{\partial \rho \tilde
s}{\partial \chi}-\nabla\cdot \frac{\partial \rho \tilde
s}{\partial
\nabla \chi}\bigg)\bs{ J}_\chi\right)\nonumber\\
& +\left(\frac{\partial \rho \tilde s}{\partial \bm{\mathcal {\bs
E}}}
-\frac{\bs P}{T}\right)
\cdot  ( \dot{ \bm{\mathcal{ E}}}+\nabla {\bs
v}
\cdot \bm{\mathcal{{\bs E}}}) + \bigg(\frac{\partial \rho \tilde
s}{\partial\bs B}-\frac{{\bs M}}{T}\bigg)\cdot({\dot{\bs B}} -{\bs
B}\cdot\nabla{\bs v})
\nonumber\\
&+ \frac{1}{T}\Bigg(\bs{\sigma}-T \frac{\partial \rho \tilde
s}{\partial\nabla\chi}\otimes\nabla\chi-T \bm{\mathcal{\bs E}}\otimes
\frac{\partial \rho \tilde s}{\partial \bm{\mathcal{\bs E}}}+
T \frac{\partial \rho \tilde s}{\partial{\bs B}}\otimes\bs B \nonumber \\
 & \hspace{3.4cm}  -
\bigg(\rho
e-\bm{\mathcal{\bs{E}}} \cdot\bs P-T\rho\tilde
s-\sum\limits_{\alpha=1}^N\rho_\alpha\mu_\alpha+\bs M\cdot\bs
B\bigg)\unity\Bigg)
:\nabla\bs{v}\nonumber\\
&+ (\bs q+
\bm{\mathcal {\bs
E}}\times{\bs M}) \cdot\nabla\frac{1}{T}
-\sum\limits_{\alpha=1}^N\bs{J}_\alpha\cdot
\Big(\nabla \, \frac{\mu_\alpha}{T}-\frac{z_\alpha}{m_\alpha
T} \bm{\mathcal{E}}\Big)-\frac{1}{T}
\sum\limits_{i=1}^{N_\textrm{R}}(R^{i}_\textrm{f}-R^{i}_\textrm{b})
\bigg(\sum\limits_{\alpha=1}^N
m_\alpha\gamma_\alpha\mu_\alpha\bigg) \nonumber\\ 
&-\bs{J}_\chi\cdot
\nabla \, \frac{\mu_\chi}{T}-\frac{1}{T}\xi_\chi\mu_\chi~.\label{UG1a}
\end{align}
Remarks on the composition and classification of terms in
\eqref{UG1a}: \\
\noindent 1. The combinations $\bm{\mathcal{\dot E}}+\nabla({\bs
v})\cdot \bm{\mathcal{{\bs E}}}$ and ${\dot{\bs B}} -{\bs
B}\cdot\nabla{\bs v}$, respectively, form objective vectors
because we have\begin{equation}\label{WD1}
    \mathcal{\dot E}_i^*+\nabla_i^*(v_k^*)\mathcal{E}_k^*=O_{ij}
    (\mathcal{\dot E}_j+\nabla_j(v_k)\mathcal{E}_k),\qquad
    {\dot B}_i^*+\nabla_k^*(v_i^*)B_k^*=\det(O)O_{ij}
    ({\dot B}_j+\nabla_k(v_j)B_k).
\end{equation}
2. The principle of material frame indifference restricts the
entropy function to the form
\begin{equation}\label{WD3}
    \rho \tilde s (\rho e- \bm{\mathcal {\bs E} }\cdot \bs P,\rho_1, \ldots
,\rho_N, \bm{ \mathcal {\bs E}},{\bs B}, \chi, \nabla \chi)=\rho \bar s
(\rho e- \bm{\mathcal {\bs E}}\cdot \bs P,\rho_1, \ldots ,\rho_N,
| \bm{\mathcal {\bs E}}|^2,|{\bs B}|^2, ( \bm{\mathcal {\bs E}}\cdot{\bs
B})^2, \chi, |\nabla \chi|^2),
\end{equation}
implying that the terms
\begin{equation}\label{WD2}
    \frac{\partial \rho \tilde
s}{\partial \nabla\chi }
\otimes\nabla\chi,\qquad - \bm{\mathcal{\bs
E}} 
\otimes \frac{\partial \rho \tilde s}{\partial\bm{\mathcal{\bs E}}}+
 \frac{\partial \rho \tilde s}{\partial{\bs B}}\otimes\bs B
\end{equation}
are symmetric objective tensors.

\noindent 3. For this reason, the factor of $\nabla \bs v$ is
symmetric and only the symmetric part $D(\bs v)$ of the velocity
gradient appears in the third line of \eqref{UG1}.

\noindent 4. Thus the representation \eqref{UG1} consists of a
divergence and a sum of binary products with objective factors of
negative, respectively positive parity.

To satisfy Axiom~(III-i), we choose the entropy flux as
\begin{equation}\label{EntF1}
    \bs \Phi=\frac{1}{T}(\bs
q + \bm{\mathcal{\bs E}} \times \bs M)-\frac{1}{T}(\sum\limits_{\alpha=1}^N\mu_\alpha
\bs{J}_\alpha+\frac{1}{\rho}\mu_\chi\bs{ J}_\chi)- \frac{\partial
\rho \tilde s}{\partial\nabla\chi}\dot\chi.
\end{equation}
Then, the remaining part of \eqref{UG1a} is identified as the
entropy production according to Axiom (IV):
\begin{align}
\zeta= &~ \left(\frac{\partial \rho \tilde s}{\partial \bm{\mathcal {\bs
E}}}-\frac{\bs P}{T}\right)\cdot( \bm{\mathcal{ \dot{\bs E}}}+\nabla({\bs
v})\cdot  \bm{\mathcal{{\bs E}}}) +
\bigg(\frac{\partial \rho \tilde
s}{\partial\bs B}-\frac{{\bs M}}{T}\bigg)\cdot({\dot{\bs B}} -{\bs
B}\cdot \nabla{\bs v})
\nonumber\\
&+ \frac{1}{T}\Bigg(\bs{\sigma}-T\frac{\partial \rho \tilde
s}{\partial\nabla\chi}\otimes\nabla\chi-T \bm{\mathcal{\bs E}}\otimes
\frac{\partial \rho \tilde s}{\partial\bm{\mathcal{\bs E}}}+T
 \frac{\partial \rho \tilde s}{\partial{\bs B}}\otimes\bs B \nonumber\\
& \hspace{3.cm}-
\bigg(\rho
e- \bm{\mathcal{\bs{E}}} \cdot\bs P-T\rho\tilde
s-\sum\limits_{\alpha=1}^N\rho_\alpha\mu_\alpha+\bs M\cdot\bs
B\bigg)\unity\Bigg)
:D(\bs{v})\nonumber\\
& +\bs q\cdot\nabla\frac{1}{T}
-\sum\limits_{\alpha=1}^N\bs{J}_\alpha\cdot
\Big(\nabla \, \frac{\mu_\alpha}{T}-\frac{z_\alpha e_0}{m_\alpha
T} \bm{\mathcal{E}}\Big)-\frac{1}{T}
\sum\limits_{i=1}^{N_\textrm{R}}(R^{i}_\textrm{f}-R^{i}_\textrm{b})
\bigg(\sum\limits_{\alpha=1}^N
m_\alpha\gamma_\alpha\mu_\alpha\bigg) \nonumber \\ &-\bs{J}_\chi\cdot
\nabla \, \frac{\mu_\chi}{T}-\frac{1}{T}\xi_\chi\mu_\chi~.\label{UG2c}
\end{align}
Each product describes a dissipative mechanism and couples a
quantity of negative parity with a quantity of positive parity.
This representation of the entropy production allows to formulate
constitutive functions for polarization, magnetization, stress,
heat flux, diffusion fluxes, reaction rates, phase flux and the
phase production rate. Cross effects between the various
dissipative mechanisms may be included. If these are
introduced by mixing within the same parity class so that the
entropy production is conserved, then we obtain the so called
Onsager symmetry as a consequence. This remarkable fact is
established and carefully described in \cite{BD13}. For illustration,  
we simply couple heat conduction and diffusion later on.

\textbf{Polarization and magnetization.} At first we discuss
constitutive equations for $\bs M$ and $\bs P$. To satisfy Axiom
III-ii we choose
\begin{equation}\label{MP1}
\bs P=T\frac{\partial \rho \tilde s}{\partial \bm{\mathcal {\bs E}}}-
\tau_{{\mathcal{E}}} ( \dot{\bm{\mathcal{{\bs E}}}}+
\nabla({\bs
v})\cdot   \bm{\mathcal{{\bs E}}})\quad\textrm{and}\quad \bs
M=T\frac{\partial \rho \tilde s}{\partial \bs B}- \tau_{
B}(\dot{\bs B}-{\bs B}\cdot\nabla({\bs v})),
\end{equation}
where $\tau_{{\mathcal{E}}}\geq0$ and $\tau_{ \ B}\geq0$ are
phenomenological coefficients. We observe that the constitutive
quantities $\bs P$ and $\bs M$ depend on the variables of the
entropy function  and, additionally, on $\nabla \bs v$ and the time
derivatives of $\bm{\mathcal{\bs{E}}}$ and $\bs B$ . These constitutive
equations embody a variety of complex phenomena, for example
hysteresis and inertia of free charges leading to frequency
dependent refraction indices. A special case arises if we set
$\tau_{{\mathcal{E}}}=0$ and $\tau_{B}=0$. Then we have the simple
constitutive equations
\begin{equation}\label{MP2}
\bs P=T\frac{\partial \rho \tilde s}{\partial \bm{ \mathcal {\bs
E}}}\quad\textrm{and}\quad \bs M=T\frac{\partial \rho \tilde
s}{\partial \bs B}
\end{equation}
that still include,  piezo-electricity, paramagnetism
and related phenomena.

\textbf{Stress.} The constitutive equation for the stress that
identically satisfies Axiom III-ii can be read off from the second line
of \eqref{UG1a}. We substitute $D(\bs v)$ by the sum of its trace
and the traceless part $D^\circ(\bs v)$. Abbreviating the factor
of $D(\bs v)$ by $\bs A$, the second line of \eqref{UG1a} reads
$1/3~\textrm{Tr}(\bs A)\textrm{div}(\bs v)+\bs A^\circ:D^\circ(\bs
v)$. Then Axiom III-ii is satisfied for the constitutive equations
\begin{equation}\label{St1}
    \frac{1}{3}\textrm{Tr}(\bs
    A)=(\lambda+\frac{2}{3}\eta)\textrm{div}(\bs
    v)\quad\textrm{and}\quad
    \bs {A}^\circ=2\eta D^\circ(\bs v).
\end{equation}
The phenomenological coefficients $\lambda+\frac{2}{3}\eta\geq0$
and $\eta\geq0$ are called bulk and shear modulus, respectively.
Thus the constitutive equation for the traceless part of the
stress reads
\begin{multline}\label{St2}
    \bs\sigma^\circ=T \bigg( \Big(\frac{\partial \rho \tilde
s}{\partial\nabla\chi}\otimes\nabla\chi- \frac{1}{3}\frac{\partial
\rho \tilde s}{\partial\nabla\chi}\cdot\nabla\chi\unity \Big)
\\ + \Big(\bm{\mathcal{\bs E}}\otimes \frac{\partial \rho \tilde
s}{\partial\bm{\mathcal{\bs E}}}-\frac{1}{3}\frac{\partial \rho \tilde
s}{\partial\bm{\mathcal{\bs E}}}\cdot\bm{\mathcal{\bs E}}\unity\Big)-
 \Big(\frac{\partial \rho \tilde s}{\partial{\bs B}}\otimes\bs B-\frac{1}{3}\frac{\partial
\rho \tilde s}{\partial\bs{B}}\cdot\bs{B}\unity \Big)\bigg) +2\eta
 D^\circ(\bs v),
\end{multline}
and for the trace of the stress we obtain
\begin{align}
\textrm{Tr}(\bs{\sigma})=& \frac{T}{3}\bigg(\frac{\partial \rho \tilde
s}{\partial\nabla\chi}\cdot\nabla\chi+\frac{\partial \rho \tilde
s}{\partial\bm{\mathcal{\bs E}}}\cdot\bm{\mathcal{\bs E}}-
 \frac{\partial \rho \tilde s}{\partial{\bs B}}\cdot\bs B\bigg)-\bigg(\rho
e- \bm{\mathcal{\bs{E}}} \cdot\bs P-T\rho\tilde
s-\sum\limits_{\alpha=1}^N\rho_\alpha\mu_\alpha+\bs M\cdot\bs B\bigg)\nonumber\\
&+(3\lambda+ 2\eta)\textrm{div}(\bs{v}).\label{St3}
\end{align}
Hence, the trace of the stress and the deviatoric stress as well
contain parts that vanish in equilibrium, viz. the terms
proportional to velocity gradients. For this reason, we prefer a
further decomposition of the stress into a so-called viscous and a
non-viscous part. We denote the viscous part by
$\bs{\sigma}^\textrm{NS}$ to refer to the Navier-Stokes system,
while the non-viscous part is simply denoted by
$\bs{\sigma}^\textrm{nv}$. The Navier-Stokes part can then be
written as
\begin{equation}\label{St4}
    \bs{\sigma}^\textrm{NS}=\lambda\textrm{div}(\bs{v})+2\eta D(\bs
    v),
\end{equation}
and for the non-viscous part we have
\begin{equation}\label{St5}
    \bs{\sigma}^\textrm{nv}=T\bigg(\frac{\partial \rho \tilde
s}{\partial\nabla\chi}\otimes\nabla\chi +\bm{\mathcal{\bs E}}\otimes
\frac{\partial \rho \tilde s}{\partial\bm{\mathcal{\bs E}}}-
 \frac{\partial \rho \tilde s}{\partial{\bs B}}\otimes\bs B\bigg)
 -\bigg(\rho
e-\bm{\mathcal{\bs{E}}} \cdot\bs P-T\rho\tilde
s-\sum\limits_{\alpha=1}^N\rho_\alpha\mu_\alpha+\bs M\cdot\bs
B\bigg)\unity.
\end{equation}

\textbf{Thermo-diffusion.} Bothe and Dreyer \cite{BD13} have
established a new method to introduce cross effects. For
illustration, we consider the coupling of heat flux and diffusion
fluxes. At first, we only consider dissipation due to diffusion.
The corresponding entropy production is
\begin{equation}\label{UG4}
    \zeta^\textrm{D}=
-\sum\limits_{\alpha=1}^N\bs{J}_\alpha\cdot
\Big(\nabla \, \frac{\mu_\alpha}{T}-\frac{z_\alpha e_0}{m_\alpha
T} \bm{\mathcal{E}}\Big)=-\sum\limits_{\alpha=1}^{N-1}\bs{J}_\alpha\cdot
\left(\nabla \, \frac{\mu_\alpha-\mu_N}{T}-\frac{1}{T} \Big(\frac{z_\alpha e_0}{m_\alpha
}-\frac{z_N e_0}{m_N }\Big) \bm{\mathcal{E}}\right) .
\end{equation}
Due to the side condition \eqref{BQ5}, constitutive equations are
only needed for $(N-1)$ fluxes. The simplest choice of
constitutive functions without coupling  is
\begin{equation}\label{D1}
    \bs{J}_\alpha = -
M_{\alpha}\left(\nabla \, \frac{\mu_\alpha-\mu_N}{T} -\frac{1}{T}
\Big(\frac{z_\alpha e_0}{m_\alpha}-\frac{z_N e_0}{m_N}\Big) \bm{\mathcal{{\bs
E}}}\right),
\end{equation}
where the mobilities $M_{\alpha}\geq0$ are non-negative
phenomenological coefficients. To introduce coupling between the
constituents of the mixture we proceed as follows. We start from
\eqref{UG4} and abbreviate for a moment the factors of the
diffusion fluxes by $\mathcal{\bs{P}}_\alpha$, i.e. we write
$$\zeta^\textrm{D}=-\sum_\alpha\bs{J}_\alpha\cdot\mathcal{\bs{P}}_\alpha.$$
Then we introduce two matrices $\bs{A}$ and $\bs{B}$ of dimension
$(N-1)^2$. We choose $\bs{A}^\textrm{T}=\bs{B}^{-1}$ and obtain
\begin{equation}\label{UG5}
   \zeta^\textrm{D}=-\sum\limits_{\alpha=1}^{N-1}\bs{J}_\alpha\cdot\mathcal{\bs{P}}_\alpha=
    -\sum\limits_{\alpha=1}^{N-1}
    \Bigg(\sum\limits_{\gamma=1}^{N-1}A_{\alpha\gamma}\bs{J}_\gamma\Bigg)
    \cdot\Bigg(\sum\limits_{\delta=1}^{N-1}B_{\alpha\delta}\mathcal{\bs{P}}_\delta\Bigg).
\end{equation}
Now we formulate constitutive equations as before, viz.
\begin{equation}\label{D2}
    \sum\limits_{\gamma=1}^{N-1}A_{\alpha\gamma}\bs{J}_\gamma=-\tilde M_\alpha
    \sum\limits_{\delta=1}^{N-1}B_{\alpha\delta}\mathcal{\bs{P}}_\delta 
\quad\textrm{with}\quad
\tilde M_{\alpha}\geq0.
\end{equation}
Solution for the diffusion fluxes yields the constitutive equation
\begin{equation}\label{D3}
    \bs{J}_\alpha=-\sum\limits_{\beta=1}^{N-1}M_{\alpha\beta}\mathcal{\bs{P}}_\beta
    \quad\textrm{with}\quad
M_{\alpha\beta}=
\sum\limits_{\gamma=1}^{N-1}B_{\gamma\alpha}\tilde
M_{\gamma}B_{\gamma\beta}.
 \end{equation}
The new mobility matrix $\bs{M}$ is positive definite and (!)symmetric. Thus, if cross
effects do not lead to additional entropy production, the Onsager symmetry is a consequence.

As a further example we now consider the entropy production
$\zeta^\textrm{HD}$ of the combined dissipative mechanisms of heat
conduction and diffusion. We have
\begin{equation}\label{UG5a}
    \zeta^\textrm{HD}=\bs{q}\cdot\nabla\frac{1}{T}
-\sum\limits_{\alpha=1}^{N-1}\bs{J}_\alpha\cdot\mathcal{\bs{P}}_\alpha.
\end{equation}
The conventional choice of constitutive functions for $(N-1)$
diffusion fluxes $(\bs{J}_{\alpha})_{\alpha=1,2,...,N-1}$ and the
heat flux $\bs{q}$ are
\begin{eqnarray}
\bs{J}_\alpha&=&-\sum\limits_{\beta=1}^{N-1}
M_{\alpha\beta} \left( \nabla \frac{\mu_\beta-\mu_N}{T}-\frac{1}{T}
\Big(\frac{z_\beta e_0}{m_\beta}-\frac{z_N e_0}{m_N}\Big)
\bm{\mathcal{{\bs E}}}\right)
+L_\alpha^{\textrm{J}}\nabla \frac{1}{T} ,\label{TD1}\\
\bs{q}&=&-\sum\limits_{\alpha=1}^{N-1}
L_\alpha^{\textrm{q}}\left(  \nabla \frac{\mu_\alpha-\mu_N}{T}-\frac{1}{T}
\Big(\frac{z_\alpha e_0}{m_\alpha}-\frac{z_N e_0}{m_N}\Big)\bm{\mathcal{{\bs
E}}}\right) +a \, \nabla \frac{1}{T} ,\label{TD2}
\end{eqnarray}
where the kinetic coefficients must satisfy the condition that the matrix 
\begin{equation}\label{TD3}
    \left(%
\begin{array}{cc}
  M_{\alpha\beta} & L^\textrm{J}_\alpha \\
  L^\textrm{q}_\alpha & a \\
\end{array}%
\right)\quad\textrm{is positive definite}.
\end{equation}
Classically, the cross effects in \eqref{TD1} and \eqref{TD2} are
related to each other by postulating the Onsager symmetry
relations
\begin{equation}\label{O1}
    M_{\alpha\beta}=M_{\beta\alpha}\quad\textrm{and}\quad
    L_\alpha^\textrm{J}=L_\alpha^\textrm{q}.
\end{equation}
According to Bothe and Dreyer \cite{BD13}, the Onsager symmetry is
achieved as follows: \\
The symmetry of the diffusion matrix is taken
to be granted from the last example. To derive the symmetry
$L_\alpha^\textrm{J}=L_\alpha^\textrm{q}$, we go back to the
entropy production \eqref{UG5} and introduce the thermo-diffusion
coefficients $D_\alpha$ by adding two terms to $\zeta^\textrm{HD}$
that conserve $\zeta^\textrm{HD}$:
\begin{equation}\label{UG5b}
    \zeta^\textrm{HD}=\left(\bs q+
    \sum\limits_{\alpha=1}^{N-1}( D_\alpha -D_N){\bs
    J}_\alpha)\right)\cdot\nabla\frac{1}{T}
-\sum\limits_{\alpha=1}^{N-1}\bs{J}_\alpha\cdot
\left(\mathcal{\bs{P}}_\alpha-(D_\alpha-D_N)\nabla\frac{1}{T}\right).
\end{equation}
Then we propose two constitutive equations for $(N-1)$ diffusion
fluxes $(\bs{J}_{\alpha})_{\alpha=1,2,...,N-1}$ and the heat flux
$\bs{q}$ by the simple relations
\begin{eqnarray}
{\bs J}_\alpha&=&-\sum\limits_{\beta=1}^{N-1}M_{\alpha\beta}
    \left(\nabla \Big( \frac{\mu_\beta}{T}-\frac{\mu_N}{T}  \Big)
    -\frac{1}{T}
    \Big(\frac{z_\beta e_0}{m_\beta}-\frac{z_N e_0}{m_N}
    \Big) \bm{\mathcal{E}} \right)  \nonumber \\
  & & \hspace{4.3cm}  
+\sum\limits_{\beta=1}^{N-1}M_{\alpha\beta} (D_\beta-D_N)
\nabla\frac{1}{T},\label{TD3a}\\
\bs{q}&=&-\sum\limits_{\alpha=1}^{N-1}(D_\alpha-D_N){\bs J}_\alpha
+a \nabla\left(\frac{1}{T}\right).\label{TD4}
\end{eqnarray}
A comparison with \eqref{TD1} and \eqref{TD2} yields
$$L_\alpha^\textrm{J}=\sum_{\beta=1}^{N-1}M_{\alpha\beta}(D_\beta-D_N)
\quad \text{and} \quad 
L_\alpha^\textrm{q}=\sum_{\beta=1}^{N-1}M_{\beta\alpha}(D_\beta-D_N).$$
Thus the symmetry $M_{\alpha\beta}=M_{\beta\alpha }$ implies
 the symmetry $L_\alpha^\textrm{J}=L_\alpha^\textrm{q}$. This
is a further example that the Onsager symmetry is a consequence of
cross effects that conserve the entropy production.

A detailed discussion of the creation of cross effects by mixing
within the parity classes is found in \cite{BD13}. For very special
cases there are various proofs of the Onsager symmetries within
the phenomenological setting. For example, Truesdell \cite{Tru69} and
M\"{u}ller \cite{Mue85} prove the symmetries under some assumptions
within a Class II model.

 \textbf{Constitutive equations for
special cases.} In this study, we are not interested in the most
general constitutive model possible. In fact, we only
consider the special case, where we have
\begin{enumerate}
\item $\tau_{{\mathcal{\bs{E}}}}=0$, $\tau_{\bs{B}}=0$ \item no
magnetization, i.e.~$\bs{M}=0$,
    \item isothermal processes,
    \item no coupling between dissipative mechanisms,
    \item the Allen-Cahn equation for the phase field, i.e. ${\bs J}_\chi=0.$
\end{enumerate}
Then the inequality \eqref{UG1} may be identically satisfied by
the simple constitutive laws
\begin{eqnarray}
\bs{P}&=&T\frac{\partial\rho \tilde s}{\partial \bm{\mathcal{\bs{E}}}},\qquad
\frac{\partial\rho \tilde s}{\partial\bs{B}}=0,\label{R0}\\
{\bs \sigma}^\textrm{NS}&=&\lambda\div({\bs v})+2\eta D(\bs v),\label{R1}\\
{\bs \sigma}^\textrm{nv}&=&T \frac{\partial \rho \tilde
s}{\partial\nabla\chi}\otimes\nabla\chi +T\bm{\mathcal{\bs E}}\otimes
\frac{\partial \rho \tilde s}{\partial\bm{\mathcal{\bs E}}}
 - \bigg(\rho
e- \bm{\mathcal{\bs{E}}} \cdot\bs P-T\rho\tilde
s-\sum\limits_{\alpha=1}^N\rho_\alpha\mu_\alpha \bigg)\unity,\label{R2}\\
 \bs{J}_\alpha&=&-\sum\limits_{\beta=1}^{N-1}
M_{\alpha\beta}\nabla\left(\frac{\mu_\beta-\mu_N}{T}-\frac{1}{T}
\Big(\frac{z_\beta e_0}{m_\beta}-\frac{z_N e_0}{m_N}\Big) \bm{\mathcal{{\bs
E}}}\right),\label{R3}\\
\ln \bigg(\frac{R_\textrm{b}^i}{R_\textrm{f}^i} \bigg)&=&\frac{1}{kT}\sum\limits_{\alpha=1}^N
m_\alpha\gamma_\alpha^i\mu_\alpha\quad\textrm{for}\quad i=1,2,...,N_\textrm{R},\label{R4}\\
\xi_\chi&=&-\tau \mu_\chi.\label{R5}
\end{eqnarray}
 The equation \eqref{R4} is an ansatz for chemical
reactions far from equilibrium. The corresponding part of the
entropy production is non-negative due to
$-\ln\big({R_\textrm{b}^i}/{R_\textrm{f}^i}\big)
(R_\textrm{f}^i-R_\textrm{b}^i)\geq
0$.\\[2mm]

\textbf{Introduction of the (Helmholtz) free energy density.}
Usually,  one prefers to have the temperature as an independent
variable instead of the internal energy density. To this end, we
introduce the (Helmholtz) free energy density 
$$\rho\psi=\rho
e-{\bs P}\cdot \bm{\mathcal{{\bs E}}} -T \rho s.$$ 
In the entropy function we change the variable 
$\rho e-{\bs P}\cdot \bm{\mathcal{{\bs
E}}}$ to $T$ with $\rho e=\rho \hat e
(T,\rho_1,...,\rho_N, \bm{\mathcal{{\bs E}}},\chi,\nabla\chi)$.
Then, for $\psi=\hat\psi(T,\rho_1,...,\rho_N, \bm{\mathcal{{\bs
E}}},\chi,\nabla \chi)$ we obtain from \eqref{EntB3} and \eqref{R0}
\begin{equation}\label{GCM2}
     \rho s=-\frac{\partial\rho\hat\psi}{\partial T},\quad
         \mu_\alpha=\frac{\partial\rho\hat\psi}{\partial\rho_\alpha},\quad
         \rho e =-T^2\frac{\partial}{\partial
     T}\left(\frac{\rho\hat\psi}{T}\right),\quad
    \bs P=-\frac{\partial\rho\hat\psi}{\partial \bm{\mathcal{\bs E}}},\quad
    \mu_\chi= \frac{\partial \rho \hat
\psi}{\partial \chi}-T\nabla\cdot \frac{1}{T}\frac{\partial \rho
\hat \psi}{\partial \nabla\chi}.
    \end{equation}
In terms of the free energy densities, the representation
\eqref{R2} of the non-viscous part of the stress reads
\begin{equation}\label{GD3}
   \bs{\sigma}^\textrm{nv}=-\frac{\partial \rho
\hat\psi}{\partial\nabla\chi}\otimes\nabla\chi+\frac{1}{2}(
\bm{\mathcal{E}}\otimes\bs P+\bs P\otimes\bm{\mathcal{{\bs E}}})+\Big(\rho\psi
-\sum\limits_{\alpha=1}^N\rho_\alpha\mu_\alpha \Big)\unity.
\end{equation}
It is convenient to introduce the pressure $p$ by
\begin{equation}\label{GD4}
    p=-\rho\psi
+\sum\limits_{\alpha=1}^N\rho_\alpha\mu_\alpha.
\end{equation}
Then, this relation is called {\it Gibbs-Duhem equation}.

The free energy density is the central constitutive quantity of a
mixture of charged and neutral constituents. Its explicit
choice is given in the next section.

\section{Choice of energy density and non-linear stability}\label{sec:coe}
We consider the quasi-static setting of electrodynamics such that $\bs B$
drops out and 
\newline $\bm{\mathcal{\bs E}}=\bs E=-\nabla \varphi$.
Moreover, we choose the following free energy density 
\begin{equation}
\label{def:innerenergy}
 \rho\psi:= W(\chi) + \frac{\gamma}{2} |\nabla \chi|^2 + h(\chi) \rho\psi_{\textrm{L}}(\rho_1,\dots,\rho_N) 
+ (1-h(\chi)) \rho\psi_{\textrm{V}}(\rho_1,\dots,\rho_N) - \frac{\veps_0}{2} s(\chi) |{\bs E}|^2,
\end{equation}
where $W(\chi):=(\chi-1)^2(\chi+1)^2$, $h:\R \rightarrow [0,1]$ is a smooth interpolation function satisfying
\begin{equation}\label{def:interpolation}
 h(z) = \left\{ \begin{array}{cl}
                 1 &\text{for} \quad z \geq 1,\\ 0&\text{for} \quad z \leq -1,
                \end{array}
\right.
\end{equation}
such that $h'(z) = 0$ for  all $ |z| \geq 1.$
In \eqref{def:innerenergy}, $\rho\psi_{\textrm{L}}, \rho \psi_{\textrm{V}}: (0,\infty)^N \longrightarrow [0,\infty)$ are the free energy 
functions of the pure phases which we assume to be given by
a combination of isotropic elastic response and entropy of mixing, i.e.,
\[\rho\psi_{\textrm{L/V}} = \sum_\alpha \rho_\alpha \psi_\alpha^{\textrm{R}}
+ (K_{L/V}-p^{\textrm{R}})\Big(1-\frac{n}{n^{\textrm{R}}}\Big)
+K_{L/V}\frac{n}{n^{\textrm{R}}} \ln 
\Big(\frac{n}{n^{\textrm{R}}}\Big) 
+ kT\sum_\alpha n_\alpha\ln \Big(\frac{n_\alpha}{ n}\Big),\]
where $K_{L/V}$ are the bulk moduli, 
 and
$n^{\textrm{R}},$ $p^{\textrm{R}},$ $\psi_\alpha^{\textrm{R}}$ are reference number density,  reference pressure and  reference energies, respectively.
Note that \eqref{def:innerenergy} and \eqref{def:interpolation} imply that $\chi=+ 1(-1)$ corresponds to liquid (vapor).
Similarly, we assume
\begin{equation}
s(\chi)= h(\chi) s_{\textrm{L}}
+ (1-h(\chi)) s_{\textrm{V}},
\end{equation}
where $s_{\textrm{L/V}}$ are the susceptibilities of the pure phases.
For brevity, we define
\begin{equation}
\begin{split}
\label{def:rhof}
 (\rho f)((\rho_\alpha)_\alpha,\chi)&:= (\rho f)(\rho_1,\dots,\rho_N,\chi)\\
&:=  h(\chi) \rho\psi_{\textrm{L}}(\rho_1,\dots,\rho_N) + (1-h(\chi))
\rho\psi_{\textrm{V}}(\rho_1,\dots,\rho_N).
\end{split}
\end{equation}
By definition, the chemical potentials are given by 
\begin{equation}
\mu_\alpha := \frac{\del(\rho \psi)}{\del \rho_\alpha} = \frac{\del (\rho f)}{\del \rho_\alpha}.
\end{equation}
This leads to the following system of equations for the partial mass densities $\rho_\alpha,$ the barycentric velocity $\bv,$
the phase field parameter $\chi$ and the electrical potential $\varphi,$ where the equation for $\rho_N$ is replaced by the evolution equation
for $\rho=\sum_{\alpha=1}^N \rho_\alpha:$  
\begin{equation}\label{eq:pde}
\begin{split}
0=&\del_t \rho_\alpha + \diver (\rho_\alpha \bv ) - 
\diver \Bigg( \betasum M_{\alpha \beta} \bigg( \nabla \frac{\mu_\beta -
  \mu_N}{T} 
+ \frac{1}{T}\bigg( \frac{z_\beta e_0}{m_\beta} - \frac{z_N e_0}{m_N}  \bigg)
\nabla \varphi \bigg) \Bigg) \\ 
& - \sum_{i=1}^{N_R} m_\alpha \gamma_\alpha^i M_{\mathrm{r}}^i
\Bigg(1-\exp\bigg( \frac{1}{kT} \sum_{\beta=1}^N m_\beta \gamma_\beta^i
\mu_\beta\bigg)\Bigg) , \qquad \qquad \qquad \alpha=1, ..., N-1,\\
 0=&\del_t \rho + \diver(\rho \bv), \\
0=& \del_t (\rho \bv) + \diver (\rho \bv \otimes \bv) + \nabla \bigg(\alphasum \rho_\alpha \mu_\alpha - \rho f -W - \frac{\gamma}{2} |\nabla \chi|^2\bigg) 
+ \gamma \, \diver(\nabla \chi \otimes \nabla \chi)
 \\&  - \diver (\sns) 
+\veps_0\diver \left((1+ s(\chi)) \left(\frac{1}{2} |\nabla \varphi|^2 \unity - \nabla \varphi \otimes \nabla \varphi \right)\right),\\
0=& \del_t\chi + \bv \cdot \nabla \chi + \frac{\tau}{\rho} \left( W' -\gamma \Delta \chi + \frac{\del (\rho f)}{\del \chi} - \frac{\veps_0}{2} s'(\chi) |\nabla \varphi|^2 \right),\\
0=&\veps_0 \, \diver ((1+s(\chi)) \nabla \varphi) + n^{\textrm{F}},
\end{split}
\end{equation}
 where $n^{\textrm{F}} := e_0\sum_{\alpha=1}^N \frac{z_\alpha}{m_\alpha} \rho_\alpha$ is the free charge density
and $M_{\textrm{r}}^i : = R_{\textrm{f}}^i.$

\subsection{Energy inequality}
With respect to the stability of the system \eqref{eq:pde}, we can prove an energy inequality.
Let $\Omega(t) \subset \R^3$ be a moving open and bounded material domain with $C^1$-boundary, and $T_f>0$ some time up to which we assume classical solutions of
\eqref{eq:pde} to exist.

  \begin{lemma}[Energy inequality]\label{lem:inequality}
 Let $((\rho_\alpha)_\alpha,\bv,\chi,\varphi)$ be a classical solution of
 \eqref{eq:pde} in $\Omega_{T_f}:=\bigcup_{t\in (0,T_f)} \Omega(t) \times \{t\}$ and let  $\varphi$ satisfy
the Laplace equation in $\R^3\setminus \Omega(t)$ for all $t \in (0,T_f).$ 
In addition, let the following boundary conditions be satisfied on $\partial \Omega(t)$ for all $t \in (0,T_f)$:
$  \bs J_\alpha=0, \nabla \chi \cdot \bs n=0,$ where $\bs n$ denotes the normal vector to $\del \Omega(t).$ Then the following
inequality holds:
\begin{equation}\label{eq:ei}
\begin{split}
& \frac{\operatorname{d}}{\operatorname{d}t}\Big( \int_\Omega W(\chi) + \frac{\gamma}{2} |\nabla \chi|^2 + (\rho f)((\rho_\alpha)_\alpha,\chi) 
+ \frac{\veps_0}{2} (1+s(\chi)) |\nabla \varphi|^2 + \frac{\rho}{2} |\bv|^2 \dx
 + \int_{\R^3\setminus \Omega} \frac{\veps_0}{2}|\nabla \varphi|^2  \dx \Big)\\
&=\int_{\del \Omega} \eps_0 \bs v \cdot\Big( \nabla \varphi_+ \otimes \nabla \varphi_+ - \nabla \varphi_- \otimes \nabla \varphi_-
  + (|\nabla \varphi_-|^2  - |\nabla \varphi_+|^2) \unity \Big) \cdot {\bs n}
  \operatorname{d} {\bs a}- 
T \int_\Omega \zeta \dx.
 \end{split}
\end{equation}
\end{lemma}

\begin{proof}
 By definition of $\rho \psi$, we may write the left hand side of \eqref{eq:ei} as
\begin{equation}
 \frac{\operatorname{d}}{\operatorname{d}t}\Big( \int_\Omega \rho e  - T \rho s + e^{\textrm{e}}+ \frac{\rho}{2} |\bv|^2 \dx
 + \int_{\R^3\setminus \Omega} e^{\textrm{e}} \dx \Big)=:\frac{\operatorname{d}}{\operatorname{d}t}A.
\end{equation}
Using transport theorems, we can express $\frac{\operatorname{d}}{\operatorname{d}t}A$ as
\begin{multline}
\frac{\operatorname{d}}{\operatorname{d}t}A = \int_\Omega 
\bigg( \Big(\rho e + \frac{\rho}{2} |\bs v|^2 + e^{\textrm{e}} \Big)_t + 
\operatorname{div} \Big((\rho e + \frac{\rho}{2} |\bs v|^2+  e^{\textrm{e}}
)\bs v \Big) \bigg)\dx\\
+ \int_{\del \Omega} ( e^{\textrm{e}}_- - e^{\textrm{e}}_+) \, \bs v \cdot \bs n
\operatorname{d} {\bs a }
- T \int_\Omega \Big( (\rho s)_t + \operatorname{div} (\rho s \bs v) \Big)\dx,
\end{multline}
where $e^{\textrm{e}}_-$, $e^{\textrm{e}}_+$ denotes the trace of $e^{\textrm{e}}$  on $\del \Omega$ from the inside and outside of $\Omega,$ respectively.
In view of the local balances \eqref{EoB3} and \eqref{A1-3}$_2$, we infer
\[ \frac{\operatorname{d}}{\operatorname{d}t}A= \int_{\del \Omega} \Big(  - \bs q \cdot {\bs n} + 
\bs v \cdot \bs \sigma \bs n + (e^{\textrm{e}}_- - e^{\textrm{e}}_+)\bs v 
\cdot \bs n + T \bs \Phi \cdot {\bs n}
 \Big) \operatorname{d} {\bs a} - T \int_\Omega \zeta \dx.
\]
Inserting the definition of $\bs \Phi$ in \eqref{EntF1} and noting that $ \bs J_\alpha=0$ on the boundary and $\bs J_\chi=0$ due to our choice of a model of Allen-Cahn type,
we find
\[ \frac{\operatorname{d}}{\operatorname{d}t} A= \int_{\del \Omega} \bs v \cdot\Big( \bs \sigma  + (e^{\textrm{e}}_- - e^{\textrm{e}}_+)\unity
 \Big)\cdot \bs n \operatorname{d}{ \bs a}
- T \int_\Omega \zeta \dx.
\]
From the momentum balance we infer
\begin{equation}
\bs \sigma^{\textrm{e}}_+ = \bs \sigma + \bs \sigma^{\textrm{e}}_-,
\end{equation}
where $\bs \sigma^{\textrm{e}}_-,\bs \sigma^{\textrm{e}}_+$ are the corresponding traces. Note, that in the quasi-static case 
\begin{equation}
 \bs \sigma^{\textrm{e}} = \eps_0 \nabla \varphi \otimes \nabla \varphi -
 e^{\textrm{e}} \unity.
\end{equation}
Therefore, $\frac{\operatorname{d}}{\operatorname{d}t}A$ becomes
\[ 
\frac{\operatorname{d}}{\operatorname{d}t}A= \int_{\del \Omega} \eps_0 \bs v \cdot\Big( \nabla \varphi_+ \otimes \nabla \varphi_+ - \nabla \varphi_- \otimes \nabla \varphi_-
  + (|\nabla \varphi_-|^2  - |\nabla \varphi_+|^2){\bf 1} \Big)\cdot \bs n 
\operatorname{d}{\bs a} - T \int_\Omega \zeta \dx.
\]
\end{proof}

\begin{remark}
 Only in case $\bs v=0$ on $\del \Omega$ the available free energy $A$ is a Lyapunov function. 
 In the current context, the admissible boundary conditions for Maxwell's equations are 
\[ \bs n \times \nabla \varphi_+ =\bs n \times \nabla \varphi_-, \quad \text{ and }\quad
 \eps_0 (1 + s(\chi_-)) \nabla \varphi_- \cdot \bs n =  \eps_0  \nabla
 \varphi_+ \cdot \bs n 
+ n^F_{\del \Omega},
\]
where $n^F_{\del \Omega}$ are free charges on $\del \Omega.$ Thus, even if $n^F_{\del \Omega}=0$ the normal component of the electric field is not continuous, in general. 

\end{remark}

\section{Non-dimensionalization}
To avoid physically meaningless scalings, we  nondimensionalize problem \eqref{eq:pde}.
To this end, we introduce reference quantities denoted by superscript $^c$ and non-dimensional quantities denoted by $^*$, i.e.,
\begin{align*}
&{\bs x}= x^c {\bs x}^*, \ t= t^c t^*, \ \rho_\alpha = \rho^c \rho_\alpha^*, \ \bv = v^c \bv^* , \ \lambda_{1,2} = \lambda^c \lambda_{1,2}^*, \ \tau = \tau^c \tau^*, \ 
M_{\alpha \beta} = M^c M_{\alpha \beta}^*, \\ & M^i_{\textrm{r}} = M_{\textrm{r}}^c (M_{\textrm{r}}^i)^*, 
\ \gamma_\beta= \gamma_{\textrm{r}}^c (\gamma_\beta^i)^*,\ m_\beta = m^c m_\beta^*, \ W= W^c W^* , \ \rho f =(\rho f)^c  (\rho f)^*, \
\gamma = \gamma^c \gamma^*,\\
 & \frac{\mu_\beta}{T} = \mu^c \mu_\beta^*, \ \varphi = \varphi^c \varphi^*, \ s = s^c s^*,\ \veps_0=  \veps_0^c \veps_0^*, \ e_0 z_\alpha = z^c z_\alpha^*.
\end{align*}
Note that $\chi$ and $h(\chi)$ do not need to be nondimensionalized and $\rho= \rho^c \rho^*$ with $\rho^*= \sum_\alpha \rho_\alpha^*$.
As we are interested in hyperbolic scalings we set $x^c = v^c t^c$ and $ (\rho f)^c = \rho^c \mu^c$ and define the following Mach and Reynolds numbers
\begin{equation}
\operatorname{M}_W := v^c \sqrt{\frac{\rho^c}{W^c}}, \quad \operatorname{M}_{\rho f} := v^c \sqrt{\frac{\rho^c}{(\rho f)^c}}, \quad
\operatorname{Re}:= \frac{\rho^c v^c x^c }{\lambda^c},
\end{equation}
as well as additional non-dimensional quantities related to the reaction and diffusion rates
\begin{equation}
\bar M_{\textrm{d}} := \frac{M^c \mu^c}{ v^c x^c \rho^c}   , \quad \bar M_{\textrm{r}} := \frac{M_{\textrm{r}}^c \gamma_{\textrm{r}}^c m^c t^c}{\rho^c},\quad
\bar A := \frac{m^c \gamma_{\textrm{r}}^c \mu^c}{kT},
 \quad \bar \tau := \frac{\tau^c t^c W^c }{\rho^c},
\end{equation}
and the electrical effects
\begin{equation}
\bar M_{\textrm{e}} := \frac{M^c z^c \varphi^c}{\rho^c m^c v^c x^c}, \quad \bar \veps:= \frac{\veps_0^c (\varphi^c)^2}{\rho^c (v^c)^2 (x^c)^2}
, \quad \underline \veps := \frac{\veps_0^c m^c \varphi^c}{(x^c)^2 z^c \rho^c}.
\end{equation}
We  assume that the small parameter  
\begin{equation}
\delta := \sqrt{\frac{\gamma^c}{(x^c)^2 W^c}}
\end{equation}
is proportional to the width of the interfacial layer. This can be justified
by $\Gamma$-limit techniques, cf. \cite{Ste88,ORS90,DK10,LM89}.
Then, suppressing $^*$ in the notation, the nondimensionalized version of \eqref{eq:pde} reads 
\begin{equation}
\begin{split}
0=&\del_t \rho_\alpha + \diver (\rho_\alpha \bv) - \diver \left(\betasum M_{\alpha \beta}
 \left( \bar M_{\textrm{d}} \nabla (\mu_\beta - \mu_N) +\bar M_{\textrm{e}}\left( \frac{z_\beta}{m_\beta} - \frac{z_N}{m_N}  \right) \nabla \varphi\right) \right)\\ 
& - \bar M_{\textrm{r}} \sum_{i=1}^{N_R} m_\alpha \gamma_\alpha^i  \bigg( 1 - 
\exp\bigg( \bar A \sum_{\beta=1}^N m_\beta \gamma_\beta^i \mu_\beta\bigg)\bigg),
 \\
0=& \del_t \rho + \diver(\rho \bv), \\
0=& \del_t (\rho \bv) + \diver (\rho \bv \otimes \bv) + \frac{1}{\operatorname{M}_{\rho f}^2}
\nabla 
\bigg(\alphasum \rho_\alpha \mu_\alpha - \rho f\bigg) 
- \frac{1}{\operatorname{M}_W^2} \nabla \big(W + \frac{\gamma}{2}\delta^2 |\nabla \chi|^2 \big)\\
+& \frac{\gamma\delta^2}{\operatorname{M}_W^2} \diver\big(\nabla \chi \otimes \nabla \chi\big)
 - \frac{1}{\operatorname{Re}}\diver (\sns) + \bar \veps\veps_0\diver \left((1+ s^cs(\chi)) \left(\frac{|\nabla \varphi|^2}{2} \unity 
- \nabla \varphi \otimes \nabla \varphi \right)\right),\\
0=& \del_t\chi + \bv \cdot \nabla \chi + \bar \tau \frac{\tau}{\rho} \left( W' -\gamma\delta^2 \Delta \chi + \frac{\operatorname{M}_W^2}{\operatorname{M}_{\rho f}^2}\frac{\del \rho f}{\del \chi}
 - \bar \veps \operatorname{M}_W^2 s^c\frac{\veps_0}{2} s'(\chi) |\nabla \varphi|^2 \right),\\
0=&\underline \veps \veps_0\diver ((1+s^c s(\chi)) \nabla \varphi) + n^F.
\end{split}
\end{equation}
In the sequel, we will consider two scaling regimes. In both of them we choose
\[ \bar A=1, \ s^c=1, \ \bar M_{\textrm{d}}=1,\ \bar M_{\textrm{r}}=1, \ \bar M_{\textrm{e}} =1, \ \operatorname{M}_W = \sqrt{\delta},
\ \operatorname{Re}= \frac{1}{\delta^2}, \ \operatorname{M}_{\rho f}=1,
\ \bar \tau = \frac{1}{\delta^2}.\]
In the {\it uncoupled regime} we consider
\begin{equation}
\bar \veps = \underline \veps =1,
\end{equation}
while we consider 
\begin{equation}
\bar \veps = \underline \veps =\delta 
\end{equation}
in the {\it coupled regime}.

\section{Sharp interface limit of the uncoupled regime \label{sec:si}}

In this section we are going to establish the sharp interface limit of the uncoupled regime, i.e., here the "small" parameter in 
the electro-static equations is not coupled to the thickness of the interfacial layer.
We use the methodology of matched asymptotic expansions. For a detailed exposition of this method we refer to e.g. \cite{Lag, CaFi}.
The treatment of a simplified version of the model at hand (without electrical
effects) can be found in \cite{DGK14}. For any quantity $f$ indexed by
$\alpha$ we will write $(f_\alpha)_\alpha$ instead of $(f_\alpha)_{\alpha=1,...\,,N}$ for brevity.

We begin by defining outer, inner and matching solutions.
The outer equations are obtained by inserting expansions of the quantities in $\delta$ into the scaled system of equations.
\begin{definition}\label{def:outersol:uc}
A tuple  $( ( \rho_{\alpha,0})_{\alpha}, \bv_0 , \chi_0, \chi_1, \varphi_0)$  with 
\begin{equation}\label{eq:outerreg:uc} \begin{split}
\rho_{\alpha,0} &\in C^0([0,T_f) , C^2(\Omega^\pm,\R_+)) \cap C^1([0,T_f) , C^0(\Omega^\pm,\R_+)),\\
\bv_0 &\in C^0([0,T_f) , C^1(\Omega^\pm,\R^3)) \cap C^1([0,T_f) , C^0(\Omega^\pm,\R^3)),\\
\chi_0 &\in  C^0([0,T_f) , C^2(\Omega^\pm,\R)),\\
\chi_1 &\in  C^0([0,T_f) , C^1(\Omega^\pm,\R)),\\
\varphi_0 &\in  C^0([0,T_f) , C^2(\Omega^\pm,\R))
\end{split}\end{equation}
 is called an {\it outer solution of the  uncoupled regime} provided 
\begin{align}
\label{eq:out1:uc} 0=&\del_{t} \rho_{\alpha,0} + \diver (\rho_{\alpha,0} \bv_0) 
- \diver \bigg(\betasum M_{\alpha \beta} \nabla (\mu_{\beta,0} - \mu_{N,0} +
\Big( \frac{z_\beta}{m_\beta} - \frac{z_N}{m_N} \Big)\varphi_0)\bigg)
\\ \nonumber & - \sum_{i=1}^{N_R} m_\alpha \gamma_\alpha^i M_{\textrm{r}}^i  \bigg(1-  \exp\bigg( \sum_{\beta=1}^N m_\beta \gamma_\beta^i \mu_{\beta,0}\bigg)\bigg), \\
\label{eq:out2:uc}0=& \del_{t }\rho_0 + \diver(\rho_0 \bv_0), \\
\label{eq:out3:uc}0=& W'(\chi_0) , \text{ in particular, }  \nabla (W(\chi_0))=0, \\
\label{eq:out4:uc}0=&\del_{t} (\rho_0 \bv_0) + \diver (\rho_0 \bv_0 \otimes \bv_0) + 
\nabla \bigg(\alphasum   \rho_{\alpha,0} \mu_{\alpha,0} - \rho f_0\bigg) 
- \nabla \big( W'(\chi_0)\chi_1 \big)   \\
\nonumber & +\veps_0 \, \diver \left((1+ s(\chi_0)) \left(\frac{1}{2} |\nabla \varphi_0|^2 \unity - \nabla \varphi_0 \otimes \nabla \varphi_0 \right)\right),\\
\label{eq:out6:uc} 
 0=& W''(\chi_0)\chi_1 + \frac{\partial \rho f}{\partial \chi}(\rho_{1,0},\dots, \rho_{N,0}, \chi_0) - \frac{\veps_0}{2} s'(\chi_0) |\nabla \varphi_0|^2,\\
\label{eq:out7:uc}0=& \veps_0\diver ((1+s(\chi_0)) \nabla \varphi_0) + \sum_{\alpha=1}^N \frac{z_\alpha}{m_\alpha} \rho_{\alpha,0} 
\end{align}
are satisfied, where we used the following abbreviations
\begin{equation}
\mu_{\alpha,0} = \mu_\alpha (\rho_{1,0},\dots, \rho_{N,0}, \chi_0) , \quad \rho f_0 = \rho f(\rho_{1,0},\dots, \rho_{N,0}, \chi_0).
\end{equation}
Note that \eqref{eq:out1:uc} holds for $\alpha=1,...\,,N-1.$
\end{definition}

The equations defining inner solutions are obtained from the scaled equations by a change of variables and inserting expansions in $\delta.$
\begin{definition}\label{def:innersol:uc}
A tuple $( (R_{\alpha,0})_{\alpha},(R_{\alpha,1})_{\alpha}, \bV_0,\Chi_0, \Chi_1, \Phi_0,\Phi_1)$ with $\Chi_0 \not\equiv 0$ and
\begin{equation}\label{eq:innerreg:uc}
\begin{split}
R_{\alpha,0} &\in C^0([0,T_f),C^0(U,C^2(\mathbb{R}_+))), \\
R_{\alpha,1} &\in C^0([0,T_f),C^0(U,C^2(\mathbb{R}))), \\
\bV_0 &\in C^0([0,T_f),C^0(U,C^1(\mathbb{R}^3))), \\
\Chi_0 &\in C^0([0,T_f),C^1(U,C^0(\mathbb{R}))) \cap C^0([0,T_f),C^0(U,C^2(\mathbb{R}))),\\
\Chi_1 &\in C^0([0,T_f),C^0(U,C^2(\mathbb{R}))),\\
\Phi_0 &\in C^0([0,T_f),C^1(U,C^0(\mathbb{R}))) \cap C^0([0,T_f),C^0(U,C^2(\mathbb{R}))),\\
\Phi_1 &\in C^0([0,T_f),C^0(U,C^2(\mathbb{R})))
\end{split}
\end{equation}
is called an {\it inner solution of the uncoupled regime} with normal velocity $\wnu$ provided
\begin{align}
\label{eq:inn0:uc} 0 &= \left(\frac{1}{2} (1+ s(\Chi_0)) |\Phi_{0,z}|^2 - (1+ s(\Chi_0)) |\Phi_{0,z}|^2 \right)_z,\\
 \label{eq:inn1:uc} 0 &= \bigg(\betasum M_{\alpha \beta} (\Mu_{\beta,0} - \Mu_{N,0})_z  
\bigg)_z \text{ for } \alpha=1,\dots,N-1,\\
 \label{eq:inn2:uc} 0 &= W'(\Chi_0) - \gamma \Chi_{0,zz}, \text{ in particular, } 0 = \bnu (- W(\Chi_0) + \frac{\gamma}{2} \Chi_{0,z}^2 )_z,\\
 \label{eq:inn4:uc} 0 &= (R_0 (\bV_0 \cdot \bnu - \wnu))_z = (j_0)_z,\\
 \label{eq:inn5:uc} 0 &= j_0 \bV_{0,z} + \bnu \bigg(\alphasum R_{\alpha,0} \Mu_{\alpha,0} - RF_0 - W'(\Chi_0)\Chi_1\bigg)_z \\ \nonumber
& \quad + \gamma \bnu (\Chi_{0,z} \Chi_{1,zz} + \Chi_{0,zz} \Chi_{1,z} - \kappa \Chi_{0,z}^2) - \nabla_\Gamma W(\Chi_0) + \gamma  \Chi_{0,zz}\nabla_\Gamma( \Chi_0 ) \\ 
\nonumber & \quad + \frac{\veps_0}{2} \left( (\Phi_{1,z})^2 + |\nabla_\Gamma \Phi_0|^2 \right) (1+s(\Chi_0))_z \bnu
  - \veps_0 ((1+s(\Chi_0)) \Phi_{1,z})_z \left(\Phi_{1,z} \bnu + \nabla_\Gamma \Phi_0  \right),
\\
 \label{eq:inn6:uc}0 &=  \frac{j_0}{\tau} \Chi_{0,z} + W''(\Chi_0) \Chi_1 -\gamma \Chi_{1,zz} + \gamma \kappa \Chi_{0,z} 
+ \!\frac{\partial R F_0}{\partial \chi} \\
&\nonumber\quad- \frac{\veps_0}{2} s'(\Chi_0) \left( (\Phi_{1,z})^2 + |\nabla_\Gamma \Phi_0|^2 \right),\\
 \label{eq:inn7:uc}0 &= \left(  (1+s(\Chi_0)) \Phi_{1,z}\right)_z,\\
\label{eq:inn3:uc} 0 &= (R_{\alpha,0}(\bV_0 \cdot \bnu - \wnu))_z - \betasum M_{\alpha \beta}\Big( (\Mu_{\beta,1} - \Mu_{N,1})_{zz}  
-\kappa (\Mu_{\beta,0} - \Mu_{N,0})_z\\
\nonumber &\quad+ ( \frac{z_\beta}{m_\beta} - \frac{z_N}{m_N}) \Phi_{1,zz}
\Big),
\end{align}
where $\bnu$ denotes the unit normal vector to the zeroth-order interface pointing into the liquid, $\nabla_\Gamma$ is the surface gradient and $\kappa$ is the mean curvature of the interface. In addition, 
we used
\begin{equation}\label{eq:abbrev}
 \begin{split}
 \Mu_{\beta,0} &:= \mu_\beta(R_{1,0},\dots, R_{N,0},\Chi_0),\  RF_0 = \rho f (R_{1,0},\dots, R_{N,0},\Chi_0),\\
 \frac{\partial RF_0}{\partial \chi}&:= \frac{\partial \rho f}{\partial \chi} (R_{1,0},\dots,R_{N,0},\Chi_0),
\quad j_0 := R_0 (\bV_0 \cdot \bnu - \wnu),\\
\Mu_{\beta,1}&:= \alphasum \frac{\del \mu_\beta}{\del \rho_\alpha}(R_{1,0},\dots,R_{N,0},X_0) R_{\alpha,1} 
+ \frac{\del \mu_\beta}{\del \chi} (R_{1,0},\dots,R_{N,0},X_0)X_1.
\end{split}
\end{equation}
Note that we have already simplified \eqref{eq:inn1:uc}--\eqref{eq:inn3:uc} using $\Phi_{0,z}=0$ which is an easy consequence of \eqref{eq:inn0:uc},
 provided the matching condition
$\Phi_{0,z}(z) \rightarrow 0$ for $z \rightarrow \pm \infty$ is satisfied.
This seems reasonable to keep the notation short(er) and is justified as we will only consider matching solutions in the sequel.
\end{definition}

Finally, we define matching solutions which consist of compatible outer and inner solutions.
\begin{definition}\label{def:matchsol:uc}
A tuple $ ( (\rho_{\alpha,0})_{\alpha}, \bv_0 , \chi_0, \chi_1, \varphi_0, (R_{\alpha,0})_{\alpha},
 (R_{\alpha,1})_{\alpha}, \bV_0,\Chi_0, \Chi_1, \Phi_0, \Phi_1)$
is called a {\it matching solution of the uncoupled regime} provided $( ( \rho_{\alpha,0})_{\alpha}, \bv_0 , \chi_0, \chi_1,\varphi_0)$
is an outer- and $((R_{\alpha,0})_{\alpha}, (R_{\alpha,1})_{\alpha},
 \bV_0,\Chi_0, \Chi_1,\Phi_0,\Phi_1)$ is an inner solution and both are linked by the standard matching conditions, see \cite{Lag,CaFi}.
\end{definition}

\begin{theorem}\label{thrm:uc}
 Let $( ( \rho_{\alpha,0})_{\alpha}, \bv_0 , \chi_0, \chi_1,\varphi_0,  (R_{\alpha,0})_{\alpha}, (R_{\alpha,1})_{\alpha},
 \bV_0,\Chi_0, \Chi_1, \Phi_0, \Phi_1)$
  be a {\bf matching solution of the uncoupled regime},
then the following equations are satisfied in the bulk regions $\Omega^\pm$:
 \begin{align}
  \label{uc:b1} \pm 1= & \chi_0, \chi_1=0, \\
  \label{uc:b2} 0=& \del_{t} \rho_{\alpha,0} + \diver (\rho_{\alpha,0} \bv_0)  
- \diver \bigg(\betasum M_{\alpha \beta} \nabla \Big(\mu_{\beta,0} - \mu_{N,0} +
\Big(\frac{z_\beta}{m_\beta}-\frac{z_N}{m_N}\Big) \varphi_0 \Big)  \bigg)\\
\nonumber & - \sum_{i=1}^{N_R} m_\alpha \gamma_\alpha^i M_{\mathrm{r}}^i\bigg(
1- \exp
\bigg(\sum_{\beta=1}^N m_\beta \gamma_\beta^i \mu_{\beta,0}\bigg)\bigg), \\
 \label{uc:b4}0=& \del_{t }\rho_0 + \diver(\rho_0 \bv_0), \\
  \label{uc:b3}0=& \del_{t} (\rho_0 \bv_0) + \diver (\rho_0 \bv_0 \otimes \bv_0) + 
\nabla \left(\alphasum   \rho_{\alpha,0} \mu_{\alpha,0} - \rho f_0 + \frac{\veps_0}{2} (1+s(\chi_0))|\nabla \varphi_0|^2 \right)\\
 \nonumber & -\left(\veps_0 (1+s(\chi_0)) \nabla \varphi_0 \otimes \nabla \varphi_0 \right), \\
  \label{uc:b5}0=&\veps_0 \diver ((1+s(\chi_0))\nabla \varphi_0) + \sum_{\alpha=1}^N \frac{z_\alpha}{m_\alpha} \rho_{\alpha,0} ,
 \end{align}
where \eqref{uc:b2} is valid for $\alpha = 1, \dots, N-1.$
 Moreover, the following conditions are fulfilled at the interface:
  \begin{align}
    \label{uc:i1} 0=& \jump{\mu_{\alpha,0} - \mu_{N,0}}  ,\\
    \label{uc:i2} 0=& \jump{\rho_0 ( \bv_0 \cdot \bnu - \wnu )}, \\
    \label{uc:i2b}0= & \jump{\rho_{\alpha,0} ( \bv_0 \cdot \bnu - \wnu )} -
\Bigg[ \! \Bigg[
\betasum M_{\alpha\beta} \nabla\Big(\mu_{\beta,0}-\mu_{N,0} +
\Big(\frac{z_\beta}{m_\beta}- \frac{z_N}{m_N}\Big) \varphi_0  \Big)\, \bnu
\Bigg]  \! \Bigg] ,\\
    \label{uc:i3} 0=& 
\Bigg[ \! \Bigg[
j_0 \bv_0 + \bigg( \alphasum   \rho_{\alpha,0}
      \mu_{\alpha,0} - \rho f_0 \bigg) \bnu +
 \veps_0 (1+s(\chi_0)) \left( \frac{1}{2}|\nabla \varphi_0|^2  -
                         \nabla \varphi_0 \otimes  \nabla \varphi_0  \right)
                         \bnu \Bigg] \! \Bigg]\\
\nonumber &\quad-
                       \gamma \kappa \bnu \int_{-\infty}^\infty  (\Chi_{0,z})^2 \dz,\\
    \label{uc:i4}0= & \Big[ \! \Big[   \frac{j_0^2}{2\rho_0^2} + \mu_{N,0}  \Big] \! \Big]
 +\frac{j_0}{\tau} \int_{-\infty}^\infty \frac{1}{R_0} (\Chi_{0,z})^2 \dz,\\
   \label{uc:i5}0= &\jump{(1+s(\chi_0))\nabla \varphi_0\cdot \bnu}, \\
      \label{uc:i6}0= &  \jump{\varphi_0},
 \end{align}
where $j_0:=\rho_0^\pm ( \bv_0^\pm \cdot \bnu - \wnu )$ and  \eqref{uc:i1}, \eqref{uc:i2b} hold for $\alpha= 1, \dots, N-1$.
Moreover, \eqref{uc:i6} implies 
\begin{equation}\label{uc:i6b}
\jump{\nabla \varphi_0 - (\nabla \varphi_0\cdot \bnu) \bnu}=0.
\end{equation}
 \end{theorem}

\begin{remark}
Note that all jump conditions in Theorem \ref{thrm:uc} are physically meaningful.
Equations \eqref{uc:i1}, \eqref{uc:i4} and  \eqref{uc:i3}  are generalised Gibbs-Thompson laws, \eqref{uc:i2} ensures conservation of mass, while \eqref{uc:i3} is a dynamic Young-Laplace law which
shows that we have surface tension of order $\delta^0.$ 
Equation \eqref{uc:i5} shows that in this scaling the electrical displacement is 
continuous across the interface,
while \eqref{uc:i6} shows the continuity of the electrical potential which causes the continuity of the tangential part of the electric field, i.e. \eqref{uc:i6b}.
 \end{remark}

We will decompose the proof of Theorem \ref{thrm:uc} into several lemmas. Our first lemma ascertains that the electrical potential is continuous across the interface.
\begin{lemma}
Any $\Phi_0$ satisfying \eqref{eq:inn0:uc} and the matching conditions, fulfils $\Phi_{0,z}=0.$
Thus, \eqref{uc:i6}  and \eqref{uc:i6b} are satisfied.
\end{lemma}
\begin{proof}
We have 
$(1+s(\Chi_0))|\Phi_{0,z}|^2 = k \in \R.$ Because of the matching condition $k=0$ and therefore $s\geq 0$ implies $\jump{\varphi_0}=0.$ 
Forming the surface gradient of $\jump{\varphi_0}=0$ gives \eqref{uc:i6b}. 
\end{proof}

Next we use the continuity of the mass flux across the interface to remove the normal velocity from the equations.

\begin{remark}\label{rem:normv}
Equation \eqref{eq:inn4:uc} is satisfied if and only if $\bV_0 \cdot \bnu = \frac{j_0}{R_0} + \wnu$
for some $j_0$ independent of $z.$
\end{remark}

The following lemma shows that we have pure phases in the bulk.

\begin{lemma}\label{lem:ChiNull}
Let $\chi_0, \chi_1$ be given as in Definition \ref{def:matchsol:uc}, then
\[ \chi_0 \in \{-1,1\} \quad \text{and} \quad \chi_1=0.\]
Furthermore, 
 all solutions 
$\Psi\in C^2(\mathbb R)$ of the ordinary differential
   equation
\begin{equation}\label{FNull}
 W'(\Psi) - \gamma \partial_{zz} \Psi =0
\end{equation}
 with 
$\partial_{z} \Psi \rightarrow 0, \Psi \to \pm 1$ as $z\rightarrow\pm\infty$ 
are given 
 by the one parameter family
\begin{equation}\label{eq:param} \Psi(z)= \bar \Psi ( z - \bar z ), \quad \bar z \in \mathbb{R},
\end{equation}
where $\bar \Psi$ is the unique monotonically increasing solution of
\eqref{FNull} satisfying $\bar \Psi(0)=0.$  In particular, all $\Chi_0$ as in 
Definition \ref{def:matchsol:uc}
are given
 by the one parameter family
\[ \Chi_0(t,\bolds,\cdot)= \bar \Psi (\cdot - \bar z(t,\bolds)), \quad \bar z \in \mathbb{R}.
\]
\end{lemma}
\begin{proof}
From \eqref{eq:out4:uc} we know $\chi_0 \in \{\pm1, 0\}$. Thus, by continuity, $\chi_0$ is constant in $\Omega^\pm.$
A phase portrait analysis which can be found in \cite{BDDJ07} shows that \eqref{FNull} implies
\eqref{eq:param} and $\chi_0^\pm= \pm 1.$
 Hence, $\chi_0 = \pm 1 $ in $\Omega^\pm$  and 
$\frac{\partial \rho f}{\partial \chi}(\rho_{1,0},\dots, \rho_{N,0}, \chi_0) =0=s'(\chi_0)$ because of \eqref{def:interpolation}. Thus,
 $\chi_1=0$ because of  $W''(\pm1) \not=0$ and equation \eqref{eq:out6:uc}. 
\end{proof}

Now we reformulate some of the equations \eqref{eq:inn1:uc}-\eqref{eq:inn3:uc} so that we obtain a system from which we can compute the $R_{\alpha,0}$ independently of $\Phi_1,\Chi_1.$

\begin{lemma}\label{lem:equivalent}
For $\Chi_0$ given as in Lemma \ref{lem:ChiNull} equations \eqref{eq:inn1:uc}, \eqref{eq:inn5:uc}, \eqref{eq:inn6:uc}, \eqref{eq:inn7:uc}
are equivalent to
\eqref{eq:inn1:uc}, \eqref{eq:inn5:uc},  \eqref{eq:inn7:uc} and
\begin{equation}\label{eq:inn8:uc}
\frac{j_0}{R_0} \left( \frac{j_0}{R_0}\right)_z  + (\Mu_{\alpha,0})_z = -\frac{j_0}{R_0\tau} (\Chi_{0,z})^2
\end{equation}
for any $\alpha =1,\dots, N$, where $R_0:= \sum_{\alpha=1}^N R_{\alpha,0}.$
\end{lemma}

\begin{proof}
Combining \eqref{eq:inn6:uc} and the normal part of \eqref{eq:inn5:uc} we get
\begin{equation}\label{eq:inn8b}
\begin{split}
 & j_0 (\bV_0\cdot \bnu)_{z} +  (\alphasum R_{\alpha,0} \Mu_{\alpha,0} - RF_0 - W'(\Chi_0)\Chi_1)_z + \gamma \Chi_{0,zz} \Chi_{1,z} \\ 
 & \quad + \frac{\veps_0}{2} \left( (\Phi_{1,z})^2 + |\nabla_\Gamma \Phi_0|^2 \right) (1+s(\Chi_0))_z 
  - \veps_0 ((1+s(\Chi_0))\Phi_{1,z})_z \Phi_{1,z}   
\\
&= - \gamma (\Chi_{0,z} \Chi_{1,zz}   - \kappa \Chi_{0,z}^2)\\
&= - \frac{j_0}{\tau} (\Chi_{0,z})^2 - W''(\Chi_0)\Chi_{0,z} \Chi_1 - \frac{\partial RF_0}{\partial \chi} \Chi_{0,z}
+ \frac{\veps_0}{2} s'(\Chi_0) X_{0,z} \left( (\Phi_{1,z})^2 + |\nabla_\Gamma \Phi_0|^2 \right).
\end{split}
\end{equation}
Because of \eqref{eq:inn7:uc} equation \eqref{eq:inn8b} implies
\begin{multline}
  j_0 (\bV_0\cdot \bnu)_{z} +  \bigg(\alphasum R_{\alpha,0} \Mu_{\alpha,0} - RF_0 - W'(\Chi_0)\Chi_1\bigg)_z + \gamma \Chi_{0,zz} \Chi_{1,z} 
\\
= - \frac{j_0}{\tau} (\Chi_{0,z})^2 - W''(\Chi_0)\Chi_{0,z} \Chi_1 - \frac{\partial RF_0}{\partial \chi} \Chi_{0,z}.
\end{multline}
This can be equivalently phrased as 
\begin{equation}
  j_0 (\bV_0\cdot \bnu)_{z} +  \bigg(\alphasum R_{\alpha,0} \Mu_{\alpha,0} - RF_0 \bigg)_z  
= - \frac{j_0}{\tau} (\Chi_{0,z})^2  - \frac{\partial RF_0}{\partial \chi} \Chi_{0,z}
\end{equation}
making use of \eqref{eq:inn2:uc}. Applying the definition of $\mu_\alpha$, this is equivalent to
\begin{equation}
  j_0 (\bV_0\cdot \bnu)_{z} +  \alphasum R_{\alpha,0} (\Mu_{\alpha,0} )_z  
= - \frac{j_0}{\tau} (\Chi_{0,z})^2.
\end{equation}
Due to \eqref{eq:inn1:uc} and the positivity of $R_0$ this implies \eqref{eq:inn8:uc}.
\end{proof}

Note that \eqref{eq:inn1:uc} and \eqref{eq:inn8:uc} form  a system of $N$ equations in which only the $R_{\alpha,0}$ are unknown, as we already know $\Chi_0$
up to a translational constant.
Thus, we may determine $(R_{\alpha,0})_\alpha$ up to the translational constant.
As the equations do not contain $\Phi_0,\Phi_1$ we can use a result from \cite{DGK14}.

\begin{lemma}\label{lem:rhoalpha}
 Let $\rho_1^-,\dots, \rho_N^->0$ be given. Let  $j_0\in \R$ with $|j_0|$ small enough,  
 $\rho^- := \sum_\alpha \rho_\alpha^-$ and $\mu_\alpha^- := \mu_\alpha(\rho_1^-,\dots,\rho_N^-,-1)$.
Then, there exist $R_{1,0},\dots,R_{N,0} \in C^0(\R,\R_+)$ and $\rho_1^+,\dots,\rho_N^+>0$ so that \eqref{uc:i1}, \eqref{uc:i4} and 
\begin{align}
0&=\mu_\alpha(R_{1,0}(z),\dots,R_{N,0}(z),\Chi_0(z)) - \mu_N(R_{1,0}(z),\dots,R_{N,0}(z),\Chi_0(z)) - \mu_\alpha^- + \mu_N^- \label{eq:ra1},\\
0&=\mu_N(R_{1,0}(z),\dots,R_{N,0}(z),\Chi_0(z)) + \frac{j_0^2}{2\sum_\alpha R^2_{\alpha,0}(z)} 
+ \frac{j_0}{\tau} \int_{-\infty}^z \frac{(\Chi_{0,z}(\tilde z))^2}{\sum_\alpha R_{\alpha,0}(\tilde z)} \operatorname{d}\tilde z \label{eq:ra2},\\
0&=\lim_{z \rightarrow \pm \infty}R_{\alpha,0}(z)-  \rho_\alpha^\pm \label{eq:ra3}
\end{align}
 are satisfied, where \eqref{eq:ra1}, \eqref{eq:ra3} are valid for $\alpha =1,\dots,N$.
In particular, the $R_{1,0},\dots,R_{N,0}$ solve 
\eqref{eq:inn1:uc} and \eqref{eq:inn8:uc}.
\end{lemma}

Thus, only the solvability criteria for $\Phi_1, \, \Chi_1,\, (R_{\alpha,1})_\alpha$ and the tangential part of $\bV_0$ are left to be determined.

\begin{lemma}\label{lem:phi1}
For $\Chi_0$ as in Lemma \ref{lem:ChiNull}, the function $\Phi_1$ from  Definition \ref{def:matchsol:uc} satisfies
\begin{equation}\label{eq:phi1} \Phi_{1,z} = \frac{k}{1+s(X_0)}\end{equation}
for some $k \in \R.$ Such an $k$ can be found if and only if
\eqref{uc:i5} holds.
\end{lemma}

\begin{proof}
As $s \geq 0 $ the equivalence of \eqref{eq:inn7:uc} to \eqref{eq:phi1} is clear.
The equivalence to the interface condition \eqref{uc:i5} follows from the matching conditions.
\end{proof}

\begin{lemma}\label{lem:linear}
The normal part of \eqref{eq:inn5:uc} can be written as 
\begin{multline}\label{eq:linear} \mathcal{L} \Chi_1  = \left(\frac{j_0^2}{\alphasum R_{\alpha,0}}\right)_z  
+ (\alphasum R_{\alpha,0} \Mu_{\alpha,0} - RF_0)_z  - \gamma  \kappa \Chi_{0,z}^2 \\- \frac{\veps_0}{2} ((1+s(\Chi_0))(\Phi_{1,z})^2)_z 
+ \frac{\veps_0}{2} |\nabla_\Gamma \Phi_0|^2 (1+s(\Chi_0))_z\end{multline}
with
\[ \mathcal{L} : W^{2,1}(\R) \rightarrow L^1(\R), \quad \Psi \mapsto ( W'(\Chi_0) \Psi - \gamma \Chi_{0,z} \Psi_z)_z.\]
Equation \eqref{eq:linear} has a solution if and only if the normal part of
\eqref{uc:i3} is true.
\end{lemma}

\begin{proof} 
The equivalence of the normal part of \eqref{eq:inn5:uc} and \eqref{eq:linear} is straightforward. Thus, we focus on the solvability condition.
The only solutions of the homogeneous adjoint problem to \eqref{eq:linear}, i.e.,
\begin{equation}\label{eq:adjoint}
W'(\Chi_0) \Xi_z + \gamma (\Chi_{0,z} \Xi_z)_z =0
\end{equation}
in $L^\infty(\R)$ are given by $\Xi(z) =k$ for all $z \in \R $ for some parameter  $k \in \R,$ see \cite[Lemma 4.11]{DGK14}.
Hence, by Fredholm's Theorem, \eqref{eq:linear} has a solution if and only if
 \begin{multline} 0= \int_{-\infty}^\infty\left(\frac{j_0^2}{\alphasum R_{\alpha,0}}\right)_z  
+ (\alphasum R_{\alpha,0} \Mu_{\alpha,0} - RF_0)_z  - \gamma  \kappa \Chi_{0,z}^2 \\- \frac{\veps_0}{2} ((1+s(\Chi_0))(\Phi_{1,z})^2)_z 
+ \frac{\veps_0}{2} |\nabla_\Gamma \Phi_0|^2 (1+s(\Chi_0))_z \dz.
\end{multline}
This is
 \begin{multline}  \bigg[ \! \bigg[ \frac{j_0^2}{\rho_0}  
+ \alphasum \rho_{\alpha,0} \mu_{\alpha,0} - \rho f_0  - \frac{\veps_0}{2}(1+s(\chi_0)) \left( (\nabla \varphi_0 \cdot \bnu)^2
-  (|\nabla \varphi_0|^2-(\nabla \varphi_0 \cdot \bnu)^2)\right)   \bigg] \!
     \bigg] \\
=  \int_{-\infty}^\infty \gamma  \kappa \Chi_{0,z}^2 \dz ,
\end{multline}
which is the normal part of \eqref{uc:i3}.
\end{proof}

Next, we show that the tangential part of \eqref{uc:i3} is a condition for the existence of the tangent part of $\bV_0$.
Let us note that due to \eqref{uc:i5}, \eqref{uc:i6b} the tangential part of \eqref{uc:i3} equals
\begin{equation}\label{uc:i3b}
j_0\jump{\bv_0 - (\bv_0 \cdot \bnu )\bnu}=0.
\end{equation}

\begin{lemma}\label{lem:tanv}
The tangential part of  \eqref{eq:inn5:uc} 
 has a solution if and only if \eqref{uc:i3b} holds.
\end{lemma}

\begin{proof}
For any vector ${\bf t}$ tangent to the zeroth order interface $\Gamma$ multiplication of \eqref{eq:inn5:uc} by ${\bf t}$ gives
\begin{equation}
 0 = j_0 (\bV_0\cdot {\bf t})_{z}    - \nabla_\Gamma W(\Chi_0)\cdot {\bf t} + \gamma  \Chi_{0,zz}\nabla_\Gamma( \Chi_0 )\cdot {\bf t}  
  - \veps_0 \left( (1+s(\Chi_0)) \Phi_{1,z}\right)_z  \nabla_\Gamma \Phi_0  \cdot {\bf t}.
\end{equation}
Due to \eqref{eq:inn2:uc} and \eqref{eq:inn7:uc} this is 
\begin{equation}
 0 = j_0 (\bV_0\cdot {\bf t})_{z},
\end{equation}
which can be solved if and only if \eqref{uc:i3b} holds.
\end{proof}

Finally, we study solvability of \eqref{eq:inn3:uc}.

\begin{lemma}\label{lem:Stefan}
 Let $((R_{\alpha,0})_{\alpha}, \bV_0, \Chi_0,\Chi_1, \Phi_1)$ be given. Then, there exist $(R_{\alpha,1})_{\alpha}$
satisfying \eqref{eq:inn3:uc}  if and only if \eqref{uc:i2b} is fulfilled.
\end{lemma}

\begin{proof}
 According to  \cite[Lemma 4.8]{DGK14}  the map
\[ (0,\infty)^N \rightarrow \R^N, \qquad 
 (\rho_\alpha)_{\alpha=1,...\,,N} \mapsto ( \mu_1((\rho_\alpha)_{\alpha=1,...\,,N},\chi), ...\,, \mu_N((\rho_\alpha)_{\alpha=1,...\,,N},\chi))^T
\]
is  a diffeomorphism for any fixed $\chi \in [-1,1]$
such that for fixed $\chi \in [-1,1]$, the matrix 
\[ \Big(\frac{\del \mu_\beta}{\del \rho_\gamma}((\rho_\alpha)_{\alpha=1,...\,,N},\chi)\Big)_{\beta,\gamma=1,...\,,N}\]
is invertible for any $\rho_1,...\,,\rho_N >0$. 

Thus, instead of studying criteria determining whether there exist functions $(R_{\alpha,1})_\alpha$ solving \eqref{eq:inn3:uc}
we may search for criteria for the existence of functions $\Mu_{\alpha,1}: \R \to \R$, $\alpha=1, ...\,,N$ satisfying \eqref{eq:inn3:uc}.
 Once we have ensured the existence of
the $( \Mu_{\alpha,1})_{\alpha}$ the  corresponding $(R_{\alpha,1})_{\alpha}$ can be computed from  \eqref{eq:abbrev}.
Conversely, if no $( \Mu_{\alpha,1})_{\alpha}$ solving \eqref{eq:inn3:uc} exist, there are no solutions in terms of $(R_{\alpha,1})_{\alpha}.$
From  \eqref{eq:inn1:uc} and the matching conditions we know  
\[   (\Mu_{\beta,0} -\Mu_{N,0})_z = 0 \quad \text{for all } \beta=1,...\,,N.\]
Thus, \eqref{eq:inn3:uc} reads
\begin{equation}
 \label{eq:inn3b} 
0 = (R_{\alpha,0}(\bV_0 \cdot \bnu - \wnu))_z - \betasum M_{\alpha \beta}\Big( (\Mu_{\beta,1} - \Mu_{N,1})_{zz}  
+ \Big( \frac{z_\beta}{m_\beta} - \frac{z_N}{m_N} \Big) \Phi_{1,zz}\Big)
\end{equation}
for $\alpha =1,\dots,N-1.$
In fact, only the differences $\psi_\beta := \Mu_{\beta,1} -\Mu_{N,1}$ for $\beta=1,...\,,N-1$  are relevant in \eqref{eq:inn3b}.
We will use the Fredholm alternative theorem to determine the solvability conditions for the $(\psi_\beta)_{\beta}.$
To this end, we choose (arbitrary) auxiliary functions $\Xi_\beta \in C^\infty(\R)$ for $\beta=1,...\,,N-1$ such that
\[ \Xi_\beta(z)= \left\{\begin{array}{lll} ( \nabla(\mu_{\beta,0} -
\mu_{N,0}))^+ \cdot \bnu z + (\mu_{\beta,1}^+ - \mu_{N,1}^+) 
&\text{ for }& z >1\\
 (\nabla(\mu_{\beta,0} - \mu_{N,0}))^- \cdot \bnu z + (\mu_{\beta,1}^- -
\mu_{N,1}^-) &\text{ for }& z <-1.
 \end{array} \right.\]
We will see that the solvability conditions are independent of the chosen auxiliary functions.
Defining $\Psi_\beta = \psi_\beta - \Xi_\beta$, we are interested in the
following auxiliary problem: Find $(\Psi_\beta)_{\beta=1,...\,,N-1} \in L^1(\R)$ such that 
\begin{equation}\label{eq:inn3c}
 \betasum (M_{\alpha\beta}(\Psi_\beta)_z)_z = (R_{\alpha,0} (\bV_0 \cdot \bnu - \wnu))_z - \betasum M_{\alpha \beta}\Big(   (\Xi_\beta)_z      + ( \frac{z_\beta}{m_\beta} - \frac{z_N}{m_N}) \Phi_{1,z}\Big)_z.
\end{equation}
The solvability conditions for \eqref{eq:inn3c} are determined by the solutions of the 
 homogenous adjoint system of equations in $L^\infty(\R).$
The homogenous adjoint system of equations reads
\begin{equation}\label{eq:adj} \betasum (M_{\alpha \beta}( Z_\beta)_z)_z
  =0\quad \text{for } \alpha=1,...\,,N-1, \quad (Z_\beta)_{\beta=1,...\,,N-1}. \quad 
\end{equation}
As the matrix $M_{\alpha \beta}$ is constant in $z$ and positive definite there are $N-1$ linearly independent solutions of \eqref{eq:adj} in $L^\infty(\R)$, 
i.e.~$\alpha=1, ...\,,N-1,$ which can be chosen as
\[ Z_\beta(z)= \delta_{\alpha \beta}, \qquad z \in \R.\]
The Fredholm alternative theorem asserts that \eqref{eq:inn3c} is solvable if and only if
\begin{equation}\label{372}
 \int_\R (R_{\alpha,0} (\bV_0 \cdot \bnu - \wnu))_z - \betasum \Big(  M_{\alpha \beta} (\Xi_\beta)_z  + ( \frac{z_\beta}{m_\beta} - \frac{z_N}{m_N}) \Phi_{1,z}   \Big)_z  \dz =0
\end{equation}
for $\alpha=1,...\,,N-1.$
Integrating \eqref{372} gives \eqref{uc:i2b}.
\end{proof}

\begin{proof}[Proof of Theorem \ref{thrm:uc}] 
We obtain the interface conditions and the properties of $\chi_0,\chi_1$ by combining the preceeding lemmas.
Concerning the bulk equations, \eqref{uc:b2} is \eqref{eq:out1:uc}, \eqref{uc:b4} is \eqref{eq:out2:uc}, \eqref{uc:b5} is \eqref{eq:out7:uc}
and \eqref{uc:b3} follows from  \eqref{eq:out4:uc} as $\chi_1=0.$

\end{proof}

\section{Sharp interface limit of the  coupled regime \label{sec:nd}  }

This section is devoted to  establishing the sharp interface limit of the coupled regime, i.e.,  the "small" parameter in 
the electro-static equations is proportional to the thickness of the interfacial layer.
As usual we begin by defining outer, inner and matching solutions.

The outer equations are obtained by inserting expansions of the quantities in $\delta$ into the scaled system of equations.
\begin{definition}\label{def:outersol:co}
A tuple  $( ( \rho_{\alpha,0})_{\alpha}, \bv_0 , \chi_0, \chi_1, \varphi_0)$  with 
\begin{equation}\label{eq:outerreg:co} \begin{split}
\rho_{\alpha,0} &\in C^0([0,T_f) , C^2(\Omega^\pm,\R_+)) \cap C^1([0,T_f) , C^0(\Omega^\pm,\R_+)),\\
\bv_0 &\in C^0([0,T_f) , C^1(\Omega^\pm,\R^3)) \cap C^1([0,T_f) , C^0(\Omega^\pm,\R^3)),\\
\chi_0 &\in  C^0([0,T_f) , C^2(\Omega^\pm,\R)),\\
\chi_1 &\in  C^0([0,T_f) , C^1(\Omega^\pm,\R)),\\
\varphi_0 &\in  C^0([0,T_f) , C^2(\Omega^\pm,\R))
\end{split}\end{equation}
 is called an {\it outer solution of the  coupled regime} provided \eqref{eq:out1:uc}, \eqref{eq:out2:uc}, \eqref{eq:out3:uc}, and
\begin{align}
\label{eq:out4:co}0=&\del_{t} (\rho_0 \bv_0) + \diver (\rho_0 \bv_0 \otimes \bv_0) + 
\nabla \bigg(\alphasum   \rho_{\alpha,0} \mu_{\alpha,0} - \rho f_0\bigg) 
- \nabla ( W'(\chi_0)\chi_1 ),  \\
\label{eq:out6:co}  
0=&  W''(\chi_0)\chi_1 + \frac{\partial \rho f}{\partial \chi}(\rho_{1,0},\dots, \rho_{N,0}, \chi_0) ,\\
\label{eq:out7:co} 0=& \sum_{\alpha=1}^N \frac{z_\alpha}{m_\alpha} \rho_{\alpha,0} 
\end{align}
are satisfied. Note that we used the following abbreviations
\begin{equation}
\mu_{\alpha,0} = \mu_\alpha (\rho_{1,0},\dots, \rho_{N,0}, \chi_0) , \quad \rho f_0 = \rho f(\rho_{1,0},\dots, \rho_{N,0}, \chi_0).
\end{equation}
\end{definition}

The equations defining inner solutions are obtained from the scaled equations by a change of variables and inserting expansions in $\delta.$
\begin{definition}\label{def:innersol:co}
A tuple $( (R_{\alpha,0})_{\alpha},(R_{\alpha,1})_{\alpha},  \bV_0,\Chi_0, \Chi_1, \Phi_0, \Phi_1)$ with $\Chi_0 \not\equiv 0$ and
\begin{equation}\label{eq:innerreg:co}
\begin{split}
R_{\alpha,0} &\in C^0([0,T_f),C^0(U,C^2(\mathbb{R}_+))), \\
R_{\alpha,1} &\in C^0([0,T_f),C^0(U,C^2(\mathbb{R}))), \\
\bV_0 &\in C^0([0,T_f),C^0(U,C^1(\mathbb{R}^3))), \\
\Chi_0 &\in C^0([0,T_f),C^1(U,C^0(\mathbb{R}))) \cap C^0([0,T_f),C^0(U,C^2(\mathbb{R}))),\\
\Chi_1 &\in C^0([0,T_f),C^0(U,C^2(\mathbb{R}))),\\
\Phi_0 &\in C^0([0,T_f),C^0(U,C^2(\mathbb{R}))),\\
\Phi_1 &\in C^0([0,T_f),C^0(U,C^2(\mathbb{R})))
\end{split}
\end{equation}
is called an {\it inner solution of the coupled regime} with normal velocity $\wnu$ provided \eqref{eq:inn4:uc} and 
\begin{align}
 \label{eq:inn4:co} 0 &= ((1+s(\Chi_0))\Phi_{0,z})_z,\\
 \label{eq:inn1:co} 0 &= \left(\betasum M_{\alpha \beta} (\Mu_{\beta,0} - \Mu_{N,0})_z + \left(\frac{z_\beta}{m_\beta} - \frac{z_N}{m_N} \right) \Phi_{0,z} 
\right)_z,\\
 \label{eq:inn2:co} 0 &= \bnu (- W(\Chi_0) + \gamma \bnu \Chi_{0,z}^2 -\frac{\veps_0}{2} (1+s(\Chi_0)) (\Phi_{0,z})^2  )_z, \\
 \label{eq:inn3:co} 0 &= W'(\Chi_0) - \gamma \Chi_{0,zz}- \frac{\veps_0}{2} s'(\Chi_0) (\Phi_{0,z})^2,\\
 \label{eq:inn5:co} 0 &= j_0 \bV_{0,z} + \bnu 
\bigg(\alphasum R_{\alpha,0} \Mu_{\alpha,0} - RF_0 - W'(\Chi_0)\Chi_1 \bigg)_z \\ \nonumber
& \quad + \gamma \bnu (\Chi_{0,z} \Chi_{1,zz} + \Chi_{0,zz} \Chi_{1,z} - \kappa \Chi_{0,z}^2) - \nabla_\Gamma W(\Chi_0) + \gamma  \Chi_{0,zz}\nabla_\Gamma( \Chi_0 ), \\
 \label{eq:inn6:co}0 &=  \frac{j_0}{\tau} \Chi_{0,z} + W''(\Chi_0) \Chi_1 -\gamma \Chi_{1,zz} + \gamma \kappa \Chi_{0,z} +
 \frac{\partial RF_0}{\partial \chi} ,\\
\label{eq:inn2b:co} 0 &= (R_{\alpha,0}(\bV_0 \cdot \bnu - \wnu))_z - \betasum M_{\alpha \beta}\Big( (\Mu_{\beta,1} - \Mu_{N,1})_{zz}  
-\kappa (\Mu_{\beta,0} - \Mu_{N,0})\\
\nonumber & \hspace{4.8cm}\quad+ \Big( \frac{z_\beta}{m_\beta} -
\frac{z_N}{m_N}\Big)
 \Phi_{1,zz} - \kappa \Big( \frac{z_\beta}{m_\beta} - \frac{z_N}{m_N}\Big) \Phi_{0,z}
\Big),\\
\label{eq:inn7:co} 0&=  ((1 + s(X_0))\Phi_{1,z})_z + \sum_{\alpha=1}^N \frac{z_\alpha}{m_\alpha} R_{\alpha,0}
\end{align}
are fulfilled, where we used the abbreviations from \eqref{eq:abbrev}.
Note that we already simplified \eqref{eq:inn5:co}, \eqref{eq:inn6:co} and \eqref{eq:inn7:co} using $\Phi_{0,z}=0$ which is an easy consequence of \eqref{eq:inn4:co},
 provided the matching condition
$\Phi_{0,z}(z) \rightarrow 0$ for $z \rightarrow \pm \infty$ is satisfied.
This seems reasonable to keep the notation short(er) and is justified as we will only consider matching solutions in the sequel.
\end{definition}

We need one more definition to state our theorem. We introduce matching solutions which consist of compatible outer and inner solutions.
\begin{definition}\label{def:matchsol:co}
A tuple $ ( (\rho_{\alpha,0})_{\alpha}, \bv_0 , \chi_0, \chi_1, \varphi_0, (R_{\alpha,0})_{\alpha},(R_{\alpha,1})_{\alpha},
 \bV_0,\Chi_0, \Chi_1, \Phi_0,\Phi_1)$
is called a {\it matching solution of the coupled regime} provided $( ( \rho_{\alpha,0})_{\alpha}, \bv_0 , \chi_0, \chi_1, \varphi_0)$
is an outer- and $((R_{\alpha,0})_{\alpha},(R_{\alpha,1})_{\alpha}, \bV_0,\Chi_0, \Chi_1,\Phi_0,\Phi_1)$
 is an inner solution and both are linked by the matching conditions.
\end{definition}

\begin{theorem}\label{thrm:co}
 Let $( ( \rho_{\alpha,0})_{\alpha}, \bv_0 , \chi_0, \chi_1,\varphi_0,  (R_{\alpha,0})_{\alpha},(R_{\alpha,1})_{\alpha},
 \bV_0,\Chi_0, \Chi_1, \Phi_0,\Phi_1)$
  be a {\bf matching solution of the coupled regime},
then the following equations are satisfied in the bulk regions $\Omega^\pm$:
 \begin{align}
  \label{co:b1} \pm1 =& \chi_0 , \chi_1=0, \\
  \label{co:b2} 0=& \del_{t} \rho_{\alpha,0} + \diver (\rho_{\alpha,0} \bv_0)
- \diver \bigg(\betasum M_{\alpha \beta} \nabla 
\Big(\mu_{\beta,0} - \mu_{N,0} + \Big(\frac{z_\beta}{m_\beta}-\frac{z_N}{m_N}
\Big) \varphi_0 \Big)\bigg)  \\
\nonumber &
- \sum_{i=1}^{N_R} m_\alpha \gamma_\alpha^i M_{\mathrm{r}}^i \bigg(1 -\exp\bigg(  \sum_{\beta=1}^N m_\beta \gamma_\beta^i \mu_{\beta,0}\bigg)\bigg),\\
 \label{co:b4}0=& \del_{t }\rho_0 + \diver(\rho_0 \bv_0) ,\\
  \label{co:b3}0=& \del_{t} (\rho_0 \bv_0) + \diver (\rho_0 \bv_0 \otimes \bv_0) + 
\nabla \left(\alphasum   \rho_{\alpha,0} \mu_{\alpha,0} - \rho f_0  \right) ,\\
  \label{co:b5}0=& \sum_{\alpha=1}^N \frac{z_\alpha}{m_\alpha} \rho_{\alpha,0} .
 \end{align}
 Moreover, the following conditions are fulfilled at the interface:
  \begin{align}
	\label{co:i1}0= & \jump{\mu_{\alpha,0} - \mu_{N,0}}  ,\\
    \label{co:i2}0= & \jump{\rho_0 ( \bv_0 \cdot \bnu - \wnu )} ,\\
    \label{co:i2b}0= & \jump{\rho_{\alpha,0} ( \bv_0 \cdot \bnu - \wnu )} -
\bigg[ \! \bigg[ \betasum M_{\alpha\beta} \nabla\Big(\mu_{\beta,0}-\mu_{N,0} +
  (\frac{z_\beta}{m_\beta}- \frac{z_N}{m_N}) \varphi_0  \Big)\cdot \bnu \bigg]
      \! \bigg],\\
    \label{co:i3}0= & \Bigg[ \! \Bigg[
j_0 \bv_0 + \left(\alphasum   \rho_{\alpha,0} \mu_{\alpha,0} - \rho f_0\right)
\bnu  \Bigg] \! \Bigg]
                      - \gamma \kappa \bnu \int_{-\infty}^\infty  (\Chi_{0,z})^2  \dz,\\
    \label{co:i4} 0=& 
\Big[ \! \Big[ \frac{j_0^2}{2\rho_0^2} + \mu_{N,0}  \Big] \! \Big] + \frac{j_0}{\tau} \int_{-\infty}^\infty \frac{1}{R_0} (\Chi_{0,z})^2 \dz,\\
   \label{co:i5} 0=&\jump{(1+s(\chi_0))\nabla \varphi_0\cdot \bnu} -
   \int_{-\infty}^\infty 
\bigg( \sum_{\alpha=1}^{N}\frac{z_\alpha}{m_\alpha} R_{\alpha,0}(z) \bigg) \dz,\\
      \label{co:i6} 0=&  \jump{\varphi_0},
 \end{align}
where $j_0:=\rho_0^\pm ( \bv_0^\pm \cdot \bnu - \wnu ) $ and $\alpha= 1,...\,, N-1$ in \eqref{co:i1} and \eqref{co:i2b}.
Moreover, \eqref{uc:i6} implies 
\begin{equation}\label{co:i6b}
\jump{\nabla \varphi_0 - (\nabla \varphi_0\cdot \bnu) \bnu}=0.
\end{equation}
 \end{theorem}

\begin{remark}
We can use the evolution equations for $(\rho_{\alpha,0})_\alpha$ together with the closure relation \eqref{co:b5} and the charge conservation 
\[\sum_\alpha \tfrac{z_\alpha}{m_\alpha}r_\alpha=0\]
 to obtain the following elliptic problem
for $\varphi_0:$
\begin{equation}
\label{eq:closure}
0=\diver\left( \sum_{\alpha,\beta=1}^{N-1} M_{\alpha\beta}\left( \left( \frac{z_\alpha}{m_\alpha} - \frac{z_N}{m_N}\right) 
\nabla (\mu_\beta - \mu_N)  +  \left( \frac{z_\alpha}{m_\alpha} - \frac{z_N}{m_N}\right)
\left( \frac{z_\beta}{m_\beta} - \frac{z_N}{m_N}\right) \nabla \varphi_0 \right)\right).
\end{equation}
Note that 
\[ \sum_{\alpha,\beta=1}^{N-1} M_{\alpha\beta}  \left( \frac{z_\alpha}{m_\alpha} - \frac{z_N}{m_N}\right)
\left( \frac{z_\beta}{m_\beta} - \frac{z_N}{m_N}\right)>0,\]
provided $\frac{z_\alpha}{m_\alpha}-\frac{z_N}{m_N}\not= 0$ for at least one $\alpha,$ due to the positive definiteness of $M_{\alpha \beta}.$
\end{remark}

The proof of Theorem \ref{thrm:co} works along the same lines as the proof of Theorem \ref{thrm:uc}. Thus, we will only highlight the differences.

\begin{lemma}\label{lem:PhiNull}
Let $\Phi_0$ be given as in Definition \ref{def:matchsol:co} then $\Phi_{0,z}=0$ and therefore $\jump{\varphi_0}=0.$
\end{lemma}
\begin{proof}
By integrating \eqref{eq:inn4:co} and using the matching conditions we obtain
\[ (1+s(\Chi_0)) \Phi_{0,z}=0.\]
Because of $s \geq 0$ this implies the claim of the lemma.
\end{proof}

Due to Lemma \ref{lem:PhiNull} equation \eqref{eq:inn3:co} simplifies to 
\begin{equation}\label{eq:inn10:co}0 = W'(\Chi_0) - \gamma \Chi_{0,zz}\end{equation}
and Lemma \ref{lem:ChiNull} applies. 
Concerning the normal velocity Remark \ref{rem:normv} is true in this regime, as well.
Moreover, \eqref{eq:inn1:co} becomes 
\begin{equation}\label{eq:inn8:co}
0 = \left(\betasum M_{\alpha \beta} (\Mu_{\beta,0} - \Mu_{N,0})_z  
\right)_z.
\end{equation}

Similarly to Lemma \ref{lem:equivalent}, we  construct a subsystem of equations from which we can compute the $(R_{\alpha,0})_\alpha.$
\begin{lemma}\label{lem:equivalent:co}
For $\Chi_0$ given as in Lemma \ref{lem:ChiNull} equations \eqref{eq:inn5:co}, \eqref{eq:inn6:co}, \eqref{eq:inn10:co}, \eqref{eq:inn8:co}
are equivalent to
\eqref{eq:inn5:co}, \eqref{eq:inn10:co},  \eqref{eq:inn8:co} and
\begin{equation}\label{eq:inn9:co}
\frac{j_0}{R_0} \left( \frac{j_0}{R_0}\right)_z  + (\Mu_{\alpha,0})_z = -\frac{j_0}{R_0\tau} (\Chi_{0,z})^2,
\end{equation}
for any $\alpha \in \{ 1,\dots,N\},$ where $R_0:= \sum_{\alpha=1}^N R_{\alpha,0}.$
\end{lemma}
\begin{proof}
The proof is analogous to the proof of Lemma \ref{lem:equivalent}.
\end{proof}

Equations \eqref{eq:inn8:co} and \eqref{eq:inn9:co} coincide with \eqref{eq:inn1:uc} and \eqref{eq:inn8:uc}. Thus, Lemma \ref{lem:rhoalpha} also applies in the scaling 
at hand.

\begin{lemma}\label{lem:linear:co}
The normal part of \eqref{eq:inn5:co} can be written as
\begin{equation}\label{eq:linear:co} \mathcal{L} \Chi_1  = \left(\frac{j_0^2}{\alphasum R_{\alpha,0}}\right)_z  
+ \bigg(\alphasum R_{\alpha,0} \Mu_{\alpha,0} - RF_0 \bigg)_z  - \gamma  \kappa \Chi_{0,z}^2 \end{equation}
with
\[ \mathcal{L} : W^{2,1}(\R) \rightarrow L^1(\R), \quad \Psi \mapsto ( W'(\Chi_0) \Psi - \gamma \Chi_{0,z} \Psi_z)_z.\]
Equation \eqref{eq:linear:co} has a solution if and only if the normal part of
\eqref{co:i3} is satisfied.
\end{lemma}
\begin{proof}
The only solutions of the homogeneous adjoint problem to \eqref{eq:linear:co}, i.e.,
\begin{equation}\label{eq:adjoint:co}
W'(\Chi_0) \Xi_z + \gamma (\Chi_{0,z} \Xi_z)_z =0
\end{equation}
in $L^\infty$ are given by $\Xi =k$ for $k \in \R,$ see \cite[Lemma 4.11]{DGK14}.
Thus, the solvability condition for \eqref{eq:linear:co} is
\[ 0= \int_{-\infty}^\infty  \left( \Bigg(\frac{j_0^2}{\alphasum R_{\alpha,0}}\Bigg)_z  
+ \bigg(\alphasum R_{\alpha,0} \Mu_{\alpha,0} - RF_0 \bigg)_z  - \gamma
\kappa \Chi_{0,z}^2 \right) \dz,\]
which is equivalent to \eqref{co:i3}.
\end{proof}

Then, we consider the tangential part of $\bV_0$. All the arguments are the same as in Lemma \ref{lem:tanv}, only the $\varphi$ and $\Phi$ terms are not present.
Thus, we obtain:
\begin{lemma}\label{lem:tanv:oc}
The tangential part of  \eqref{eq:inn5:co} 
 has a solution if and only if the tangential part of \eqref{co:i3} holds.
\end{lemma}
Integrating \eqref{eq:inn7:co} leads to the statement:
\begin{lemma}\label{lem:phi1:co}
 Let $(R_{\alpha,0})_\alpha$ and $X_0$ be given. Then, it exists a solution $\Phi_1$ of \eqref{eq:inn7:co} if and only if \eqref{co:i5} holds. 
\end{lemma}
Finally, analogous to Lemma \ref{lem:Stefan}, we obtain the following result:
\begin{lemma}\label{lem:Stefan:co}
Equation  \eqref{eq:inn2b:co} 
 has a solution if and only if  \eqref{co:i2b} holds.
\end{lemma}
\begin{proof}[Proof of Theorem \ref{thrm:uc}] 
We obtain the interface conditions and the properties of $\chi_0,\chi_1$ by combining the preceeding lemmas.
Concerning the bulk equations, \eqref{co:b2} is \eqref{eq:out1:uc}, \eqref{co:b4} is \eqref{eq:out2:uc}, \eqref{co:b5} is \eqref{eq:out7:co}
and \eqref{co:b3} follows from  \eqref{eq:out4:co} as $\chi_1=0.$
\end{proof}

\section*{Acknowledgment}
The authors would like to thank Clemens Guhlke for fruitful 
discussions and advices on some topics related to the paper.

\bibliographystyle{plain}
\bibliography{refs}
\end{document}